\documentclass[10pt,leqno]{article}

\usepackage{amsfonts}
\usepackage{amsmath}
\usepackage{amssymb}
\newcommand{\rarrow}[1]{{\buildrel #1 \over \longrightarrow}}

\def\Ker{{\rm Ker }}
\def\Z{{\mathbf Z}}

\def\H{{\mathbf H}}

\def\P{{\rm P}}

\makeatletter
  \@addtoreset{equation}{section}
\makeatother
\makeatletter

\renewenvironment{thebibliography}[1]
      {\begin{center}\normalfont\Large\bfseries\refname\end{center}%
       \@mkboth{\MakeUppercase\refname}{\MakeUppercase\refname}%
       \list{\@biblabel{\@arabic\c@enumiv}}%
            {\settowidth\labelwidth{\@biblabel{#1}}%
             \leftmargin\labelwidth
             \advance\leftmargin\labelsep
             \@openbib@code
             \usecounter{enumiv}%
             \let\p@enumiv\@empty
             \renewcommand\theenumiv{\@arabic\c@enumiv}}%
       \sloppy
       \clubpenalty4000
       \@clubpenalty \clubpenalty
       \widowpenalty4000%
       \sfcode`\.\@m}
      {\def\@noitemerr
        {\@latex@warning{Empty `thebibliography' environment}}%
       \endlist}
\makeatother

\newtheorem{thm}{Theorem}[section]

\newtheorem{thm0}[thm]{Theorem}
\newtheorem{cor}[thm]{Corollary}
\newtheorem{prop}[thm]{Proposition}

\newtheorem{lem}[thm]{Lemma}

\newtheorem{conj}[thm]{Conjecture}
\newenvironment{proof}{\noindent{\bf Proof.}}
{\noindent \ $\Box$ \bigskip}
\begin{document}

\title{RELATIONS IN THE $24$-TH HOMOTOPY GROUPS OF SPHERES}
\author{\textsc{Toshiyuki} MIYAUCHI and \textsc{Juno} MUKAI}
\date{}
%\subjclass{Primary 55Q40; Secondary 55Q50}

%\keywords{homotopy groups of spheres, Toda bracket, $J$-morphism}

\maketitle
\begin{abstract}
The main purpose of this note is to give a proof of the fact that the Toda brackets \ $\langle\bar{\nu},\sigma,\bar{\nu}\rangle$ and $\langle\nu,\eta, \bar{\sigma}\rangle$ are not trivial. This is an affirmative answer of the second author's Conjecture (Determination of the $P$-image by Toda brackets, Geometry and Topology Monographs \textbf{13}(2008), 355--383). The second purpose is to show the relation $\bar{\nu}_7\omega_{15}=\nu_7\sigma_{10}\kappa_{17}$ in $\pi^7_{31}$.
\end{abstract}
%\noindent
\section{Introduction}
Throughout this note, we work in the $2$-primary components of homotopy groups
of spheres. In the stable group, let \ $\iota\in\pi^s_0, \ \eta\in\pi^s_1,\ \nu\in\pi^s_3,\ \sigma\in\pi^s_7,\ 
\bar{\nu},\ \varepsilon\in\pi^s_8,\ \mu\in\pi^s_9,\ \zeta\in\pi^s_{11},\ \kappa\in\pi^s_{14},\ \rho\in\pi^s_{15},\ \omega,\ \eta^*\in\pi^s_{16},\ \bar{\mu}\in\pi^s_{17},\ \nu^*,\ \xi\in\pi^s_{18},\ \bar{\zeta},\ \bar{\sigma}\in\pi^s_{19},\ \bar{\kappa}\in\pi^s_{20},\ \bar{\rho}\in\pi^s_{23},\ \delta\in\pi^s_{24},\ \mu_3\in\pi^s_{25}$\ be the generators \cite{T, MMO}. 

We know the following \cite{MMO}: \

$$
\pi^s_{24}= \Z_2\{\bar{\mu}\sigma\}
\oplus\Z_2\{\eta\eta^*\sigma\}.
$$ 

The main purpose of this note is to give a proof of the fact that the Toda brackets \ $\langle\bar{\nu},\sigma,\bar{\nu}\rangle$ and $\langle\nu,\eta, \bar{\sigma}\rangle$ are not trivial. 

\begin{thm0} \label{thm}\label{main}
$\eta_9\sigma_{10}\eta^*_{17}=\{\varepsilon_9,\sigma_{17},\eta_{24}\sigma_{25}\}_4$,\ 
$\{\bar{\nu}_{20},\sigma_{28}, \bar{\nu}_{35}\}=\{\nu_{20},\eta_{23}, \bar{\sigma}_{24}\}=\eta_{20}\sigma_{21}\eta^*_{28}=\eta_{20}\eta^*_{21}\sigma_{37}$ \ and \
$\langle\bar{\nu},\sigma, \bar{\nu}\rangle=\langle\nu,\eta, \bar{\sigma}\rangle=\eta\eta^*\sigma$. 
\end{thm0}

This result gives an affirmative answer to \cite[Conjecture 4.8]{Mu2}. 
In the proof of Theorem \ref{main}, our method is to inspect relations in homotopy groups of spheres through those in homotopy groups of rotation groups. 

We know the following \cite[Theorem 12.22]{T}:
$$
\pi^{13}_{31}=\Z_8\{\xi_{13}\}\oplus\Z_8\{\lambda\}\oplus\Z_2\{\eta_{13}\bar{\mu}_{14}\}.
$$ 

Let $P: \pi^{13}_{32}\to\pi^6_{30}$ be the $P$-homomorphism\ ($P=\Delta$ in \cite{T}). We need 
\begin{lem}\label{HPxl13}
$H(P\xi_{13})\equiv\xi'\ \bmod\ 2\lambda',2\xi'$\ and\  
$H(P\lambda)\equiv\lambda' \bmod 2\lambda',2\xi'$. 
\end{lem}
Notice that Lemma \ref{HPxl13} improves \cite[(3.3)]{MMO}.

Oda \cite[Proposition 2.6 (5)]{Od} obtained the following relation in $\pi^7_{31}$: 
$$
\bar{\nu}_7\omega_{15}\equiv 0\ \bmod\ \nu_7\sigma_{10}\kappa_{17},  \bar{\zeta}'_7,
$$ 
where $\bar{\zeta}'_7=\sigma'\varepsilon_{14}\mu_{22}$  \cite[(5.10)]{MMO}.
The second purpose of this note is to show
\begin{thm0} \label{thm}\label{main2}
$\bar{\nu}_6\omega_{14}\equiv\nu_6\sigma_9\kappa_{16}+P(\xi_{13}+\lambda)\circ\eta_{29}\ \bmod\ 4\bar{\zeta}'_6$\ and\\ $\bar{\nu}_7\omega_{15}=\nu_7\sigma_{10}\kappa_{17}$.
\end{thm0}

The authors wish to thank Professor Nobuyuki Oda for 
kind advices during the preparation of the manuscript.

\section{Recollection of some relations in homotopy groups of spheres}
We use the result,
the notation of \cite{T} and the properties of Toda brackets freely. We know 
\begin{equation}\label{n0}
[\iota_2,\iota_2]=2\eta_2,
\end{equation}
\begin{equation}\label{n1}
\pm[\iota_4,\iota_4]=2\nu_4-E\nu',
\end{equation}
\begin{equation}\label{n2}
\eta_4\nu_5=(E\nu')\eta_7=[\iota_4,\eta_4],\ \eta_5\nu_6=0, 
\end{equation}
\begin{equation}\label{n3}
\nu_5\eta_8=[\iota_5,\iota_5],
\end{equation} 
\begin{equation}\label{n4}
2\sigma_8-E\sigma'=\pm[\iota_8,\iota_8]
\end{equation} 
and  \cite[Lemma 6.3]{T}
\begin{equation}\label{etbn}
\eta_5\bar{\nu}_6=\nu^3_5 \text{ and } 
\bar{\nu}_6\eta_{14}=\nu^3_6. 
\end{equation} 
We also know that
\begin{equation}\label{et7sg}
\eta_7\sigma_8=\bar{\nu}_7+\varepsilon_7+\sigma'\eta_{14}
\text{\ \ \cite[(7.4)]{T}},
\end{equation}
\begin{equation}\label{et9sg}
\eta_9\sigma_{10}=\bar{\nu}_9+\varepsilon_9 
\text{\ \ \cite[Lemma 6.4]{T}},
\end{equation}
 and $\eta_9\sigma_{10}+\sigma_9\eta_{16}
=[\iota_9,\iota_9]$ \cite[(7.1)]{T}.
We recall from \cite[(7.19)]{T} the relation 
\begin{equation}\label{sg'nu}
\sigma'\nu_{14}=x\nu_7\sigma_{10}\ (x: \mbox{odd}).
\end{equation}
%Using Proposition 2.5 of \cite{T}, we have the following.
By (\ref{sg'nu}) and \cite[Lemma 5.14, (7.21)]{T}, we have
\begin{equation}\label{nu10sg}
\nu_{10}\sigma_{13}=2\sigma_{10}\nu_{17}=[\iota_{10},\eta_{10}]
\end{equation}
and
\begin{equation}\label{sgm11n}
\sigma_{11}\nu_{18}=[\iota_{11},\iota_{11}].
\end{equation}
By \cite[(7.18), (7.22)]{T}, we have relations
%\begin{equation}\label{nu6ep}
$$
\nu_6\varepsilon_9=(2\bar{\nu}_6)\nu_{14}=[\iota_6,\nu^2_6]
$$
%\end{equation}
and
\begin{equation}\label{bnun}
\bar{\nu}_9\nu_{17}=[\iota_9,\nu_9].
\end{equation}

We recall the relations:
%\begin{equation}\label{4ze5}
$$
4\zeta_5=\eta^2_5\mu_7,\text{ and } 4\bar{\zeta}_5=\eta^2_5\bar{\mu}_7
\text{\ \ \cite[(7,14), Lemma 12.4]{T}},
$$
%\text{\ \ (\cite[(7,14), Lemma 12.4]{T})},
%\end{equation}
\begin{equation}\label{et4zt}
\eta_4\zeta_5\equiv\ (E\nu')\mu_7\ \bmod\ (E\nu')\eta_7\varepsilon_8
\text{\ \ \cite[Proposition 2.2(5)]{Og}},
\end{equation}
\begin{equation}\label{bnu6sg} 
\varepsilon_3\sigma_{11}=0 \text{ and }
\bar{\nu}_6\sigma_{14}=0.
\text{\ \ \cite[Lemma 10.7]{T}},
\end{equation}
$$
[\iota_{17},\iota_{17}]\equiv\eta^*_{17}+\omega_{17}\ \bmod\ \sigma_{17}\mu_{24}
\text{\ \ \cite[Proposition 12.20.ii)]{T}},
$$
\begin{equation}\label{n7k}
\nu_7\kappa_{10}=\kappa_7\nu_{21}
\text{\ \ \cite[Proposition 2.13(2)]{Og}},
\end{equation}
\begin{equation}\label{n*19x}
[\iota_{19},\iota_{19}]=\nu^*_{19}+\xi_{19}
\text{\ \ \cite[Corollary 12.25]{T}},
\end{equation}
\begin{equation}\label{n9xi}
\nu_9\xi_{12}=\sigma^3_9 \text{ and }  \xi_{12}\nu_{30}=\sigma^3_{12}
\text{\ \ \cite[II-Proposition 2.1(2)]{Od}},
\end{equation}
\begin{equation}\label{bep}
\eta_6\kappa_7=\bar{\varepsilon}_6,\ \kappa_9\eta_{23}=\bar{\varepsilon}_9
\text{\ \ \cite[(10.23)]{T}},
\end{equation}
\begin{equation}\label{epep}
\nu_5\sigma_8\nu^2_{15}=\eta_5\bar{\varepsilon}_6,\ \varepsilon^2_3
 =\varepsilon_3\bar{\nu}_{11}=\eta_3\bar{\varepsilon}_4=\bar{\varepsilon}_3\eta_{18}
\text{\ \ \cite[Lemma 12.10]{T}},
\end{equation}
%\begin{equation}\label{mepbn}
$$
\mu_3\varepsilon_{12}\equiv\eta_3\mu_4\sigma_{13}\ \bmod\ 2\bar{\varepsilon}'
\text{\ \ \cite[Proposition 2.13(7)]{Og}},
$$
%\end{equation}
and
\begin{equation}\label{mepbn2}
\mu_5\bar{\nu}_{14}=0 \text{ and } \bar{\nu}_6\mu_{14}=0
\text{\ \ \cite[Proposition 2.13(8)]{Og}}.
\end{equation}

By  \cite[Lemma 9.2]{T} and the relation $8\sigma^2_{14}=0$, we have 
\begin{equation}\label{nuzt}
\nu_7\zeta_{10}=(E^2\sigma''')\sigma_{14}\ \mbox{and} \ \nu_{14}\zeta_{17}=0.
\end{equation}

By \cite[Proposition 2.2(6)]{Og} and \cite[(7.25)]{T}, we have
\begin{equation}\label{nu9m}
\zeta_6\eta_{17}=\nu_6\mu_9=8([\iota_6,\iota_6]\sigma_{11}).
\end{equation}

%\begin{lem}\label{tz7te} $\{\zeta_9,\eta_{20},2\iota_{21}\}=0$.\end{lem}
%\begin{proof} We know $\pi^9_{22}=\{\sigma_9\nu^2_{16}\}\cong\Z_2$. Since $\zeta_9\circ\pi^{20}_{22}=\{\zeta_9\eta^2_{20}\}=0$ and $2\pi^9_{22}=0$, we have 
%$${\rm Ind}\{\zeta_9,\eta_{20},2\iota_{21}\}= \zeta_9\circ\pi^{20}_{22}+2\pi^9_{22}=0.$$
%By (\ref{sgm11n}), $\nu_9\sigma_{12}\nu_{15}=0$, we obtain 
%$$0=E^2\{\zeta_7,\eta_{18},2\iota_{19}\}\subset\{\zeta_9,\eta_{20},2\iota_{21}\}\ \bmod\ 0$$ and complete the proof.\end{proof}

By the relations \eqref{et9sg}
and $\bar{\nu}_9\sigma_{17} = \varepsilon_9\sigma_{17} =0$ \cite[Lemma 10.7]{T},
 we have
\begin{equation}\label{et9sg2}
\eta_9\sigma^2_{10}= (\bar{\nu}_9+\varepsilon_9)\sigma_{17}
	=\bar{\nu}_9\sigma_{17} + \varepsilon_9\sigma_{17} =0. 
\end{equation}

We recall from \cite[(10.18), (10.20)]{T} that 
$\sigma_9(\bar{\nu}_{16}+\varepsilon_{16})=[\iota_9,\sigma_9]$ , 
$\sigma_{10}\bar{\nu}_{17}=\sigma_{10}\varepsilon_{17}=[\iota_{10},\nu^2_{10}]$
and 
\begin{equation}\label{n5}
\sigma_{11}\bar{\nu}_{18}=\sigma_{11}\varepsilon_{18}=0.
\end{equation}

By (\ref{epep}) and (\ref{sgm11n}), we have
\begin{equation}\label{n6}
\eta_9\bar{\varepsilon}_{10}=0.
\end{equation}

We recall the elements $\lambda'$ and $\xi'$ in $\pi^{11}_{29}$ from
\cite[Lemma 12.19]{T} and \cite[Proposition 4(3)]{Od1}: 
\begin{equation}\label{2lambda}
E^2\lambda'=2\lambda,\ %
H(\lambda')=\varepsilon_{21}
\end{equation}
and
\begin{equation}\label{2xi}
E^2\xi'=2\xi_{13},\ %
H(\xi')= \bar{\nu}_{21}+\varepsilon_{21}=\eta_{21}\sigma_{22}.
\end{equation}

By \eqref{2lambda}, \eqref{2xi}
and the proof of \cite[Proposition 2.20(6)]{Og}, 
we obtain \begin{equation}\label{n11om}
\nu_{11}\omega_{14}=(\lambda'+\xi')\eta_{29}\text{\ and\ } 
\nu_{13}\omega_{16}=(2\lambda+2\xi_{13})\eta_{31}=0.
\end{equation}

We also recall \cite[(12.23)]{T}
\begin{equation}\label{n7}
4\sigma_{10}\zeta_{17}=2\sigma_{11}\zeta_{18}
 =\sigma_{13}\zeta_{20}=\zeta_{13}\sigma_{24}=0.
\end{equation}

By  (\ref{n7}), \cite[Theorem 12.8, (12.25), p. 166]{T} and using the EHP sequence, we obtain 
%We recall \cite[Lemmas 12.14, 12.19, Theorem 12.22]{T}
\begin{equation}\label{sgm12zt}
4\sigma_9\zeta_{16}=[\iota_9,\eta_9\mu_{10}],\ 2\sigma_{10}\zeta_{17}=[\iota_{10},\mu_{10}],\ \sigma_{12}\zeta_{19}=8[\iota_{12},\sigma_{12}].
\end{equation}
 
By \cite[Proposition 2.17(7)]{Og} and (\ref{n7}), we have
\begin{equation}\label{nu11ro} 
\nu_{10}\rho_{13}\equiv 0 \bmod 2\sigma_{10}\zeta_{17}\ \ \text{and}
\ \ \nu_{11}\rho_{14}=0. 
\end{equation}

By the relations 
$\bar{\zeta}_5\eta_{24}\equiv \zeta_5\mu_{16}\equiv \nu_5\bar{\mu}_8\
	\bmod\ \nu_5\eta_8\mu_9\sigma_{18}$
\cite[II-Proposition 2.2(1);(2)]{Od}
 and
$\nu_6\bar{\mu}_9=16(P\rho_{13})$
\cite[(16.6)]{MT}, we have
%\begin{equation}\label{nu6bm}
$$
\bar{\zeta}_6\eta_{25}=\zeta_6\mu_{17}=\nu_6\bar{\mu}_9=16(P\rho_{13}).
$$
%\end{equation}

We recall from \cite[Theorem 12.6, (12.4)]{T} that :
$$
\pi^7_{23}=\{\sigma'\mu_{14}, \sigma'\eta_{14}\varepsilon_{15},
\mu_7\sigma_{16}, \eta_7\bar{\varepsilon}_8\} 
\ (E\zeta'=\sigma'\eta_{14}\varepsilon_{15}).
$$
By the relations $E^2\sigma'=2\sigma_9$ (\ref{n4}), $2\eta_n=0$ (\ref{n0}) and $2\mu_n=0$ for $n\ge 3$ \cite[Lemma 6.5]{T}, we have $E^{\infty}(\sigma'\mu_{14})=0$ and
$E^{\infty}(\sigma'\eta_{14}\varepsilon_{15})=0$.
We know $E^{\infty}(\eta_7\bar{\varepsilon}_8)=0$ (\ref{n6}) and  
$E^{\infty}(\mu_7\sigma_{16})\ne 0$ \cite[Theorem 12.16]{T}.
Hence we have 
\[
E^\infty\pi^7_{23}=\{\sigma\mu\}.
\]
We observe that  
$\langle\zeta,\eta,\nu\rangle\circ\eta
	=\zeta\circ\langle\eta,\nu,\eta\rangle=\zeta\nu^2=\zeta\nu\circ \nu=0$ (\ref{nuzt}). 
Since $\pi^s_{16}=\{\sigma\mu, \eta^*\}\cong(\Z_2)^2$ and ${(\eta)}^*: \pi^s_{16}\to\pi^s_{17}$ is a monomorphism by \cite[Theorem 12.16, 12.17]{T},
we have $\langle\zeta,\eta,\nu\rangle=0$. 
Hence we obtain the relation
\begin{equation}\label{tze7tet}
\mu_7\sigma_{16}\not\in\{\zeta_7,\eta_{18},\nu_{19}\}.
\end{equation}

Denote by ${\rm Ind}\{ \ , \ ,\  \}_k$ the indeterminacy of the Toda bracket. 
We show 
\begin{lem}\label{zetep}
$\zeta_5\bar{\nu}_{16}=0$\ and\  $\zeta_5\varepsilon_{16}=\zeta_5\eta_{16}\sigma_{17}=
\nu_5\mu_8\sigma_{17}$.
\end{lem}
\begin{proof}
Since $\bar{\nu}_{16}= \{\nu_{16}, \eta_{19}, \nu_{20}\}_1$
\cite[Lemma 6.2]{T} and $\pi^{19}_{24}=0$ \cite[Proposition 5.9]{T},
we have 
$$
\zeta_5\bar{\nu}_{16}\in\zeta_5\circ\{\nu_{16},\eta_{19},\nu_{20}\}_1
\subset\{\zeta_5\nu_{16},\eta_{19},\nu_{20}\}_1\ \bmod\ \pi^5_{21}\circ\nu_{21}.
$$
By the fact that $\pi^5_{21}=\{\mu_5\sigma_{14},\eta_5\bar{\varepsilon}_6\}$ 
\cite[Theorem 12.6]{T} and the relation $\sigma_{12}\nu_{19}=0$ \cite[(7.20)]{T},
 we obtain $\pi^5_{21}\circ\nu_{21}=\{\eta_5\bar{\varepsilon}_6\nu_{21}\}$. 
By the relations (\ref{bep}), (\ref{n7k}) and $\eta_6\nu_7=0$ (\ref{n2}), we have
\begin{equation}\label{ot}
\eta_5\bar{\varepsilon}_6\nu_{21}=\eta_5\eta_6\kappa_{7}\nu_{21}
	=\eta_5\eta_6\nu_7\kappa_{10}=0.
\end{equation}
So, we get that ${\rm Ind}\{\zeta_5\nu_{16},\eta_{19},\nu_{20}\}_1=
\pi^5_{21}\circ\nu_{21}=0$. 

We recall the relation \cite[Proposition 2.4(2)]{Og}:
$$
\zeta_5\nu_{16}\equiv \nu_5\zeta_8\ \bmod\ \nu_5\bar{\nu}_8\nu_{16}.
$$
By the relation $\bar{\nu}^2_5=0$ \cite[Proposition 2.8(2)]{Og}, we have 
$$
\{\nu_5\bar{\nu}_8\nu_{16},\eta_{19},\nu_{20}\}_1\supset\nu_5\bar{\nu}_8\circ\{\nu_{16},\eta_{19},\nu_{20}\}_1\ni\nu_5\bar{\nu}^2_8=0\ \bmod\ 
\pi^5_{21}\circ\nu_{21}=0.
$$
Hence, we obtain 
$$
\{\zeta_5\nu_{16},\eta_{19},\nu_{20}\}_1=\{\nu_5\zeta_8,\eta_{19},\nu_{20}\}_1.
$$
We observe that
$$
\{\nu_5\zeta_8,\eta_{19},\nu_{20}\}_1\supset\nu_5\circ\{\zeta_8,\eta_{19},\nu_{20}\}_1\ \bmod\ \pi^5_{21}\circ\nu_{21}=0.
$$
%We obtain $\langle\zeta,\eta,\nu\rangle\subset\pi^s_{16}=\{\mu\sigma,\eta^*\}\cong(\Z_2)^2$ and $\eta\circ\pi^s_{16}\cong(\Z_2)^2$. By \cite[Lemmas 5.5, 9.2]{T}, $\eta\circ\langle\zeta,\eta,\nu\rangle=\langle\eta,\zeta,\eta\rangle\circ\nu=\nu^2\zeta=0$. This implies $$ \langle\zeta,\eta,\nu\rangle=0. $$
By the fact that
$\pi^8_{21}=\{\sigma_8\nu^2_{15},\nu_8\sigma_{11}\nu_{18}\}\cong(\Z_2)^2$
\cite[Theorem 7.7]{T} and the relation \eqref{epep}, we obtain 
$$
E\{\zeta_7,\eta_{18},\nu_{19}\}\subset\{\zeta_8,\eta_{19},\nu_{20}\}_1\ \bmod\ \pi^8_{21}\circ\nu_{21}=\{\sigma_8\nu^3_{15},\eta_8\bar{\varepsilon}_9\}.
$$
So, by (\ref{tze7tet}), we have
$$
\{\zeta_8,\eta_{19},\nu_{20}\}_1\subset\{\sigma_8\nu^3_{15},\eta_8\bar{\varepsilon}_9,(E\sigma')\mu_{15},(E\sigma')\eta_{15}\varepsilon_{16}\}.
$$
By the relations \eqref{epep} and \eqref{ot}, we have
$\nu_5\sigma_8\nu^3_{15}=\eta_5\bar{\varepsilon}_6\nu_{21}=0$.
Hence, by (\ref{sg'nu}) and $2\nu_5\sigma_8=\nu_5(E\sigma')$
\cite[(7.16)]{T}, we see that
$\nu_5\circ\{\zeta_8,\eta_{19},\nu_{20}\}_1=0$. This leads to the first half. 

From the fact that $\bar{\nu}_{16}=\varepsilon_{16}+\eta_{16}\sigma_{17}$ and the first half, we obtain $\zeta_5\varepsilon_{16}=\zeta_5\eta_{16}\sigma_{17}$. 
The last equality is just \cite[Proposition 2.20(11)]{Og}. This leads to the second half and completes the proof. 
\end{proof}

We recall the relation \cite[Lemma 12.13]{T}:
\begin{equation}\label{16}
16\bar{\sigma}_6=\nu_6\mu_9\sigma_{18}.
\end{equation}
By the group structures of $\pi^k_{k+19}$ for $k=6,7$ and  the arguments in \cite[p. 151-3]{T} including \cite[Lemma 12.13, (12.17)]{T}, we obtain
\begin{equation}\label{2bsg}
2\bar{\sigma}_6=yP(\sigma^2_{13}) \ (y: \mbox{odd}).
\end{equation}

By Lemma \ref{zetep} and (\ref{2bsg}), we have
\begin{equation}\label{ze6ep}
\zeta_6\varepsilon_{17}=\zeta_6\eta_{17}\sigma_{18}=16\bar{\sigma}_6
\text{ and }
\zeta_7\varepsilon_{18}=0.
\end{equation}

We show 
\begin{lem}\label{sgnep}
{\rm (1)} $\langle\sigma,\nu,\sigma\rangle=\xi$. \\
{\rm (2)} $\langle\sigma,\nu,\eta\sigma\rangle
 =\langle\sigma,\nu,\varepsilon\rangle=\langle\nu,\sigma,\varepsilon\rangle
 =\langle\nu,\varepsilon,\sigma\rangle=0$.
\end{lem}
\begin{proof}
Firstly, we recall that $\pi_{11}^S=\{\zeta\}$, $\pi_{12}^S=0$, $\pi_{16}^S=\{\eta\rho, \omega\}$ and the relations 
$\eta\nu=0$ (\ref{n2}), $\sigma\zeta=0$ (\ref{n7}), $\varepsilon\zeta=0$ (\ref{ze6ep}), $\nu\omega=0$ (\ref{n11om}).
%\cite[Theorem 14.1]{T}.
The indeterminacies of all brackets are $0$, because
$\sigma\circ\pi^s_{11}=\pi^s_{11}\circ\sigma=0$ and 
$$
\sigma\circ\pi^s_{12}=\pi^s_{12}\circ\sigma=\pi^s_{11}\circ\eta\sigma=\pi^s_{11}\circ\varepsilon
=\nu\circ\pi^s_{16}=0
$$

The relation (1) follows directly from the definition of $\xi_{12}$ \cite[p.153]{T}. 

By the fact that $\langle\sigma,\nu,\eta\rangle\subset \pi^S_{12}=0$,
we have 
$$
\langle\sigma,\nu,\eta\sigma\rangle\supset\langle\sigma,\nu,\eta\rangle\circ\sigma=0
\ \bmod\ 0.%\sigma\circ\pi^s_{12}+\pi^s_{11}\circ\eta\sigma=0.
$$

By the Jacobi identity of Toda brackets \cite[(3.7)]{T}, the definition of $\varepsilon$
\cite[(6.1)]{T}, the relation 
$\langle\nu,\eta,2\iota\rangle\subset\pi_5^S=0$ and $\pi_{13}^S=0$, we obtain
\[\begin{split}
\langle\sigma,\nu,\varepsilon\rangle
 =\langle\sigma,\nu,\langle\eta,2\iota,\nu^2\rangle\rangle
 &\equiv\langle\sigma,\langle\nu,\eta,2\iota\rangle,\nu^2\rangle
  +\langle\langle\sigma,\nu,\eta\rangle,2\iota,\nu^2\rangle\\
 &=\langle\sigma,0,\nu^2\rangle
  +\langle0,2\iota,\nu^2\rangle\\
&\ni 0\ \bmod\ \sigma\circ\pi^s_{12}+\pi^s_{13}\circ\nu^2=0.
\end{split}\]

By the relations 
$\langle\eta\sigma,\nu,\sigma\rangle
 =\langle\nu,\sigma,\bar{\nu}\rangle=\bar{\sigma}$
\cite[I-Proposition 3.3 (4)]{Od}, \eqref{et9sg} and using \cite[(3.9).i)]{T},
we have 
\[\begin{split}
\langle\nu,\sigma,\varepsilon\rangle=\langle\nu,\sigma,\bar{\nu}+\eta\sigma\rangle
 &\subset \langle\nu,\sigma,\bar{\nu}\rangle+ \langle\nu,\sigma,\eta\sigma\rangle\\
 &=\langle\nu,\sigma,\bar{\nu}\rangle+ \langle\eta\sigma,\nu,\sigma\rangle
 =\{\bar{\sigma}+\bar{\sigma}\}=0.
\end{split}
\]
By the relations $\langle\sigma,\nu,\varepsilon\rangle=0$, 
$\langle\nu,\sigma,\varepsilon\rangle=0$ and use of \cite[(3.9).ii), (3.10)]{T}, we have
\[
\begin{split}
0&\in\langle\nu,\varepsilon,\sigma\rangle+\langle\varepsilon,\sigma,\nu\rangle
  +\langle\sigma,\nu,\varepsilon\rangle\\
&=\langle\nu,\varepsilon,\sigma\rangle+\langle\nu,\sigma,\varepsilon\rangle
  +\langle\sigma,\nu,\varepsilon\rangle\\
 &=\langle\nu,\varepsilon,\sigma\rangle\ \bmod 0.
\end{split}
\]  
This leads the last equality of (2) and completes the proof.
\end{proof}

We show 
\begin{cor}\label{sgnep1}
$\{\nu_{11},\sigma_{14},\varepsilon_{21}\}\subset\{\lambda'\eta_{29}, \xi'\eta_{29}\}$ and $E^2\{\nu_{11},\sigma_{14},\varepsilon_{21}\}=0$.
\end{cor}
\begin{proof}
We recall 
$
\pi^{11}_{30}=\{\lambda'\eta_{29}, \xi'\eta_{29}, \bar{\zeta}_{11}, \bar{\sigma}_{11}\}
$ \cite[Theorem 12.23]{T}.
By stabilizing this result, Lemma \ref{sgnep} (2) and using the fact that 
$\{\bar{\zeta}_{11}, \bar{\sigma}_{11}\}\cong\pi^s_{19}$, we obtain
$$
\{\nu_{11},\sigma_{14},\varepsilon_{21}\}\subset\{\lambda'\eta_{29}, \xi'\eta_{29}\}.
$$
This leads to the first half.

The second half is obtained from the first half, \eqref{2lambda} and \eqref{2xi}.
This completes the proof.
\end{proof}

\section{Proof of Lemma \ref{HPxl13}}

We recall the relation \cite[I-Proposition 3.1(1)]{Od}
\begin{equation}\label{nulm}
\nu_{10}\lambda=\sigma_{10}\kappa_{17}. 
\end{equation}

By (\ref{n7k}) and (\ref{nulm}), we obtain 
\begin{equation}\label{nulmn}
\nu_{10}\lambda\nu_{31}=\sigma_{10}\nu_{17}\kappa_{20}. 
\end{equation}

By \cite[II-Proposition 2.1(2)]{Od}, 
\begin{equation}\label{k7s}
\kappa_7\sigma_{21}=0.
\end{equation}

By \cite[Proposition 1(4)]{Od1}, 
\begin{equation}\label{2xi"}
2\xi''\equiv \sigma_{10}\zeta_{17},\ \bmod\ 2\sigma_{10}\zeta_{17}. 
\end{equation}

Next we recall the element $\omega'\in\pi^{12}_{31}$ \cite[Lemma 12.21, (12.27), p. 166]{T}: 
\begin{equation}\label{e2omg}
E^2\omega'=2\omega_{14}\nu_{30}=[\iota_{14},\nu^2_{14}] \text{\ and\ }
H(\omega')\equiv\varepsilon_{23}\ 
\bmod\ \varepsilon_{23}+\bar{\nu}_{23}.
\end{equation}

By \cite[(7.30), p. 166]{T}, we have
\begin{equation}\label{x13et}
\xi_{13}\eta_{31}=[\iota_{13},\sigma_{13}]
=(E\theta)\sigma_{25}.
\end{equation}

By the relations 
%\begin{equation}\label{ldet31}
%\lambda\eta_{31}\equiv E\omega'\ \bmod\ \xi_{13}\eta_{31},
%\end{equation}
$%\begin{equation}\label{n13et*}
\nu_{13}\eta^*_{16}\equiv E\omega'\ \bmod\ \xi_{13}\eta_{31}
$%\end{equation}
\cite[Proposition 2.20(8)]{Og},
(\ref{e2omg}) and (\ref{x13et}), 
we have 
\begin{equation}\label{n14et*}
\nu_{14}\eta^*_{17}=[\iota_{14},\nu^2_{14}].
\end{equation}

We show the following lemma overlapping with \cite[Lemma 2.18]{Og}:

\begin{lem}\label{usgnep}
\begin{description}
\item{\rm (1)}
$\{2\sigma_{11},\nu_{18},\sigma_{21}\}\equiv\xi'\ \bmod\ 2\lambda', 2\xi'$.
\item{\rm (2)}
$\{2\sigma_{11},\nu_{18},\varepsilon_{21}\}=\lambda'\eta_{29}$.
\item{\rm (3)}
$\xi_{12}\eta_{30}\equiv\theta\sigma_{24}\ \bmod\ [\iota_{12},\eta_{12}\sigma_{13}]$. 
\item{\rm (4)}
$\{\sigma_{12},\nu_{19},\varepsilon_{22}\}\ni\omega'+a\xi_{12}\eta_{30}
 \bmod\  (E\lambda')\eta_{30},(E\xi')\eta_{30}$ for $a\in\{0,1\}$.
\end{description}
\end{lem}
\begin{proof}
Since $2\sigma_{11}\circ\pi^{18}_{30}=0$, we have
$$
\{2\sigma_{11},\nu_{18},\sigma_{21}\}=\{2\sigma_{11},\nu_{18},\sigma_{21}\}_1.
$$
By using \cite[Proposition 2.6]{T} and \eqref{nu10sg}, we obtain
$$
H\{2\sigma_{11},\nu_{18},\sigma_{21}\}_1
 =-P^{-1}(2\sigma_{10}\nu_{17})\circ\sigma_{22}=\{\eta_{21}\sigma_{22}\}.
$$
This means $H\{2\sigma_{11},\nu_{18},\sigma_{21}\}_1=H\xi'$ from \eqref{2xi}
and hence, 
$\{2\sigma_{11},\nu_{18},\sigma_{21}\}_1\ni\xi'\ \bmod\ E\pi^{10}_{28}
 =E\{\lambda'',\xi'',\eta_{10}\mu_{11}\}=\{2\lambda',2\xi',\eta_{11}\bar{\mu}_{12}\}$. 
By stabilizing this relation, we obtain 
$$
\{2\sigma_{11},\nu_{18},\sigma_{21}\}_1\ni\xi'\ \bmod\ 2\lambda', 2\xi'.
$$ 
This leads to (1). 

By the fact that $\pi^{17}_{29}=0$, $\pi^{18}_{30}=0$, $\pi^{11}_{22}=\{\zeta_{11}\}$ \cite[Theorem 7.4 and 7.6]{T} and 
\eqref{ze6ep}, we have 
\[
{\rm Ind}\{2\sigma_{11},\nu_{18},\varepsilon_{21}\}_1
 =2\sigma_{11}\circ E\pi^{17}_{29}+\pi^{11}_{22}\circ\varepsilon_{22}=0.
\]
By using \cite[Proposition 2.6]{T} and \eqref{nu10sg}, we obtain
$$
H\{2\sigma_{11},\nu_{18},\varepsilon_{21}\}_1=-P^{-1}(2\sigma_{10}\nu_{17})\circ\varepsilon_{22}=\{\eta_{21}\varepsilon_{22}\}.
$$
Since $H(\lambda'\eta_{29})=\varepsilon_{21}\eta_{29}=\eta_{21}\varepsilon_{22}$
from \eqref{2lambda}, 
we have $\{2\sigma_{11},\nu_{18},\varepsilon_{21}\}_1\ni\lambda'\eta_{29}\ \bmod\ E\pi^{10}_{29}=\{\bar{\zeta}_{11},\bar{\sigma}_{11}\}$. This leads to (2).

Next we show (3). We recall $\pi^{12}_{23}=\{\zeta_{12}, [\iota_{12},\iota_{12}]\}$. 
%Since $\sigma_{12}\circ\pi^{19}_{31}=\pi^{12}_{23}\circ\varepsilon_{23}=0$, we have ${\rm Ind}\{\sigma_{12},\nu_{19},\varepsilon_{22}\}=\{\sigma_{12},\nu_{19},\varepsilon_{22}\}_1=0$. That is, the bracket  $\{\sigma_{12},\nu_{19},\varepsilon_{22}\}=\{\sigma_{12},\nu_{19},\varepsilon_{22}\}_1$ consists of a single element. We obtain 
%$$H\{\sigma_{12},\nu_{19},\varepsilon_{22}\}_1=-P^{-1}(\sigma_{11}\nu_{18})\circ\varepsilon_{23}=\{\eta_{22}\varepsilon_{23}\}.$$
%By the fact that $H(\omega')=\varepsilon_{23}$,  $E\pi^{11}_{30}=\{(E\lambda')\eta_{30},(E\xi')\eta_{30},\bar{\sigma}_{11},\bar{\zeta}_{11}\}$ and Lemma \ref{sgnep}, we obtain (5). 
By the definitions of $\xi_{12}$ and $\theta$ \cite[p153, Lemma 7.5]{T}, we have
\begin{eqnarray*}
\xi_{12}\eta_{30}
&\in&\{\sigma_{12},\nu_{19},\sigma_{22}\}_1\circ\eta_{30}\\
&\subset&\{\sigma_{12},\nu_{19},\sigma_{22}\eta_{29}\}_1\\
&=&\{\sigma_{12},\nu_{19},\eta_{22}\sigma_{23}\}_1\\
&\supset&\{\sigma_{12},\nu_{19},\eta_{22}\}_1\circ\sigma_{24}\\ 
&\ni&\theta\sigma_{24}\ \bmod\ \sigma_{12}\circ E\pi^{18}_{30}+\pi^{12}_{23}\circ\eta_{23}\sigma_{24}.
\end{eqnarray*}
By \cite[Lemma 7.4 and 7.6]{T} and \eqref{ze6ep},
 we get that 
$\sigma_{12}\circ E\pi^{18}_{30}=0$ and 
$\pi^{12}_{23}\circ\eta_{23}\sigma_{24}
 =\{\zeta_{12}, [\iota_{12},\iota_{12}]\}\circ\eta_{23}\sigma_{24}
 =\{[\iota_{12},\eta_{12}\sigma_{12}]\}$. 
This leads to (3).

We have $\pi^{12}_{23}\circ\varepsilon_{23}=\{[\iota_{12},\varepsilon_{12}]\}$. 
By the fact that $H(\lambda')=\varepsilon_{21}$ \eqref{2lambda} and the argument in \cite[Lemma 4.3]{GM}, we obtain
$$
{\rm Ind}\{\sigma_{12},\nu_{19},\varepsilon_{22}\}_1={\rm Ind}\{\sigma_{12},\nu_{19},\varepsilon_{22}\}=\{(E\lambda')\eta_{30}\}.
$$
By using \cite[Proposition 2.6]{T}, we obtain
\begin{equation}\label{Hsg12}
H\{\sigma_{12},\nu_{19},\varepsilon_{22}\}=-P^{-1}(\sigma_{11}\nu_{18})\circ\varepsilon_{23}=\{\varepsilon_{23}\}.
\end{equation}
By the relations $H(\xi_{12}\eta_{30})=H(\xi_{12})\eta_{30}=\sigma_{23}\eta_{30}$
\cite[Lemma 12.14]{T}, \eqref{et7sg}
 and \eqref{e2omg}, 
we have
$H(\omega'+a\xi_{12}\eta_{30})=\varepsilon_{23}$ for $a=0$ or $a=1$.
Hence we have
$
H(\alpha)=H(\omega'+a\xi_{12}\eta_{30})
$
for any representative $\alpha$ in $\{\sigma_{12},\nu_{19},\varepsilon_{22}\}$ and
by the fact that 
$E\pi^{11}_{30}=\{(E\lambda')\eta_{30},(E\xi')\eta_{30},
 \bar{\zeta}_{12},\bar{\sigma}_{12}\}$ \cite[Theorem 12.23]{T}, we have
$$
\alpha-(\omega'+a\xi_{12}\eta_{30})\in\{(E\lambda')\eta_{30},  (E\xi')\eta_{30}, \bar{\zeta}_{12}, \bar{\sigma}_{12}\}.
$$
By stabilizing this relation and using the relation $\langle\sigma, \nu, \varepsilon\rangle=0$ (Lemma \ref{sgnep}), the fact that $\{\bar{\zeta}_{12}, \bar{\sigma}_{12}\}\cong\pi^s_{19}$, we obtain
$$
\alpha-(\omega'+a\xi_{12}\eta_{30})\in\{(E\lambda')\eta_{30},  (E\xi')\eta_{30}\}.
$$
This leads to (4) and completes the proof. 
\end{proof}

We recall \cite[II-Proposition 2.1(6), III-Proposition 2.2(2)]{Od}:
\begin{equation}\label{x'sm}
\xi'\sigma_{29}\equiv \sigma_ {11}\nu^*_{18}\ \bmod\ 2\sigma_ {11}\nu^*_{18}
\end{equation}
and
\begin{equation}\label{lmdsg}
\lambda\sigma_{31}=0.
\end{equation}

Since $H(\lambda'\sigma_{29})=\varepsilon_{21}\circ\sigma_{29}=0$ from \eqref{2lambda}
and \cite[Lemma 10.7]{T},  
$E\pi^{10}_{35}=\{\sigma_{11}\xi_{18},\sigma_{11}\nu^*_{18},\mu_{3,11},\eta_{11}\bar{\mu}_{12}\sigma_{29}\}\cong(\Z_4)^2\oplus(\Z_2)^2$, $\pi^s_{25}=\{\mu_3,\eta\bar{\mu}\sigma\}\cong(\Z_2)^2$ \cite[Theorem 1(a)]{Od} and
$E^2(\lambda'\sigma_{29})=2\lambda\sigma_{31}=0$,  we have
\begin{equation}\label{ld'sgm}
\lambda'\sigma_{29}\in\{\sigma_{11}\xi_{18}, \sigma_ {11}\nu^*_{18}\}.  
\end{equation}

%We recall from \cite[I-(2.1), Proposition 6.4(9)]{Od} that the relation 
By the fact that $E\theta'=[\iota_{12},\eta_{12}]$ \cite[(7.30)]{T} and (\ref{bep}), 
\begin{equation}\label{n8}
[\iota_{12},\bar{\varepsilon}_{12}]=(E\theta')\kappa_{24}. 
\end{equation}

By \cite[Theorem 1(b), I-Proposition 3.5(6)]{Od} and Lemma \ref{usgnep}(3), 
\begin{equation}\label{smbsm}
\xi_{12}\eta_{30}\sigma_{31}=\theta\sigma^2_{24}=\sigma_{12}\bar{\sigma}_{19}
\text{\ \ and\ \ }
[\iota_{13},\sigma^2_{13}]=\sigma_{13}\bar{\sigma}_{20}\ne 0.
\end{equation}

We show 
\begin{lem}\label{omg'sg}
$\omega'\sigma_{31}=0$.
\end{lem}
\begin{proof}
Using the EHP sequence and the relation
$$
H(\omega'\sigma_{31})=H(\omega')\sigma_{31}
 \equiv\varepsilon_{23}\sigma_{31}=0
\ \bmod \ (\varepsilon_{23}+\bar{\nu}_{23})\sigma_{31}=0
$$
from \eqref{e2omg} and \eqref{bnu6sg}, we have
$\omega'\sigma_{31}\in E\pi^{11}_{37}$. 
By \eqref{e2omg} and \cite[(7.20)]{T}, we have 
$E^2(\omega'\sigma_{31})=2\omega_{14}\nu_{30}\sigma_{33}=0$.
Hence, \cite[I-Theorem 1(b), (8.19)]{Od} implies that there exist $b, c\in\{0,1\}$
satisfing the equation
$$
\omega'\sigma_{31}=4bE\tau'''+c[\iota_{12},\bar{\varepsilon}_{12}]\
$$
and
$$
E(\omega'\sigma_{31})=4bE^2\tau'''=32bE\tau^{IV}=0.
$$
By the EHP sequence, we have
$$
\omega'\sigma_{31}\in P\pi^{25}_{40}=\{[\iota_{12},\iota_{12}]\circ\rho_{23}, [\iota_{12},\bar{\varepsilon}_{12}]\}.
$$
By \cite[Theorem 1(b), (8.19)]{Od}, the order of 
$[\iota_{12},\iota_{12}]\circ\rho_{23}=32\tau^{IV}$ is $32$. This and (\ref{n8}) imply
$$
\omega'\sigma_{31}\in\{32\tau^{IV}, (E\theta')\kappa_{24}\}.
$$

On the other hand, by Lemma \ref{usgnep}(4), we obtain
$$
(\omega'+a\xi_{12}\eta_{30})\sigma_{31}\in\{\sigma_{12},\nu_{19},\varepsilon_{22}\}\circ\sigma_{31}\ \bmod\ (E\lambda')\eta_{30}\sigma_{31},(E\xi')\eta_{30}\sigma_{31}.
$$
By \cite[Proposition 1.4]{T}, we have 
$$
\{\sigma_{12},\nu_{19},\varepsilon_{22}\}\circ\sigma_{31}=-\sigma_{12}\circ\{\nu_{19},\varepsilon_{22},\sigma_{30}\}.
$$
By the fact that $\pi^{19}_{38}\cong\pi^s_{19}$ and by Lemma \ref{sgnep}(2), 
we obtain $\{\sigma_{12},\nu_{19},\varepsilon_{22}\}\circ\sigma_{31}=0$. 
By (\ref{ld'sgm}), (\ref{x13et}), \cite[Proposition 2.20(3)]{Og} and (\ref{x'sm}),
we have 
$$
(E\lambda')\eta_{30}\sigma_{31}=(E\lambda')\sigma_{30}\eta_{37}\in\{\sigma_{12}\xi_{19}, \sigma_{12}\nu^*_{19}\}\circ\eta_{37}=0
$$
and
$$
(E\xi')\eta_{30}\sigma_{31}=(E\xi')\sigma_{30}\eta_{37}\equiv\sigma_{12}\nu^*_{19}\eta_{37}=0\ \bmod\ 2\sigma_{12}\nu^*_{19}\eta_{37}=0.
$$
Hence, we have $(\omega'+a\xi_{12}\eta_{30})\sigma_{31}=0$ and 
the relation (\ref{smbsm}) implies
$$
\omega'\sigma_{31}=a\sigma_{12}\bar{\sigma}_{19}.
$$
Finally, $32\tau^{IV}, (E\theta')\kappa_{24}$ and $\sigma_{12}\bar{\sigma}_{19}$ are independent in $\pi^{12}_{38}$ \cite[I-Theorem 1(b)]{Od}. This completes the proof. 
\end{proof}

By this lemma and its proof, we obtain
\begin{equation}\label{omg*}
\omega'\in\{\sigma_{12},\nu_{19},\varepsilon_{22}\}\ \bmod\ (E\lambda')\eta_{30},(E\xi')\eta_{30}.
\end{equation}

By \cite[Theorem 10.10]{T}, \eqref{nu11ro}, \eqref{bep} and \cite[(3.6)]{GM},
we have $\nu_{13}\circ\pi^{16}_{31}
 =\nu_{13}\circ\{\rho_{16}, \bar{\varepsilon}_{16}, [\iota_{16},\iota_{16}]\}=0$.
By \cite[Theorem 7.4]{T} and \eqref{n7}, we have
$\pi^{13}_{24}\circ\sigma_{24}=\{\zeta_{13}\sigma_{24}\}=0$.
Hence, we obtain 
$$
{\rm Ind}\{\nu_{13},2\sigma_{16},\sigma_{23}\}=
\nu_{13}\circ\pi^{16}_{31}+\pi^{13}_{24}\circ\sigma_{24}=0.
$$
Then, by \cite[I-Proposition 3.4(8)]{Od}, we have
\begin{equation}\label{tn13sg}
\{\nu_{13},2\sigma_{16},\sigma_{23}\}=\{\nu_{13},\sigma_{16},2\sigma_{23}\}=\xi_{13}+x(\lambda+2\xi_{13})
\end{equation}
for some odd integer $x$.

We show 
\begin{lem}\label{omg'}
\begin{description}
\item[\rm (1)]
$H(\omega')=\varepsilon_{23}$.
\item[\rm (2)] $\lambda\eta_{31}=E\omega'=\{\sigma_{13},\nu_{20},\varepsilon_{23}\}_n\ (n\le 13)$.
\item[\rm (3)] $\nu_{13}\eta^*_{16}=E\omega'+\xi_{13}\eta_{31}$.
\end{description}
\end{lem}
\begin{proof}
(1) follows from \eqref{omg*} and \eqref{Hsg12}.

By \cite[Theorem 7.4 and 7.6]{T} and \eqref{ze6ep}, we have
$$
{\rm Ind}\{\sigma_{13},\nu_{20},\varepsilon_{23}\}_n
=\sigma_{13}E^n\pi^{20-n}_{32-n}+\pi^{13}_{24}\circ\varepsilon_{24}=0
\ (n\le 13).
$$

By the relations (\ref{lmdsg}), 
$\lambda\eta_{31}\equiv E\omega'\ \bmod\ \xi_{13}\eta_{31}$
\cite[Proposition 2.20(2)]{Og}, 
%(\ref{x13et}), 
Lemma \ref{omg'sg} and (\ref{smbsm}), 
we have
$$
0=\lambda\sigma_{31}\eta_{38}=\lambda\eta_{31}\sigma_{32}\equiv (E\omega')\sigma_{32}=0\ \bmod\ \xi_{13}\eta_{31}\sigma_{31}=[\iota_{13},\sigma^2_{13}]\ne 0.
$$
This leads to the first equality of (2). The second of (2) follows from (\ref{omg*}), \eqref{2lambda} and \eqref{2xi}. 
 
By the definitions of $\sigma^*_{16}$  and $\eta^*_{16}$ \cite[p. 153]{T}, we obtain 
$$
\sigma^*_{16}\eta_{38}\in\{\sigma_{16},2\sigma_{23},\sigma_{30}\}\circ\eta_{38}\subset\{\sigma_{16},2\sigma_{23},\sigma_{30}\eta_{37}\}\supset\{\sigma_{16},2\sigma_{23},\eta_{30}\}\circ\sigma_{32}
$$
$$
\ni\eta^*_{16}\sigma_{32}\ \bmod\ \sigma_{16}\circ\pi^{23}_{39}+\pi^{16}_{31}\circ\eta_{31}\sigma_{32}.
$$
Since $\sigma_{16}\eta_{23}\rho_{24}
 =\sigma_{16}\mu_{23}\sigma_{32}=\rho_{16}\eta_{31}\sigma_{32}=\mu_{16}\sigma^2_{25}=0$ 
by \cite[Propsition 12.20]{T} and \cite[(2.9)]{MMO},
$\bar{\varepsilon}_{16}\eta_{31}\sigma_{32}=\eta_{16}\bar{\varepsilon}_{17}\sigma_{32}=0$
by \eqref{n6},% and {$[\iota_{16},\iota_{16}]\eta_{31}\sigma_{32}=0$}, 
we have
$\sigma_{16}\circ\pi^{23}_{39}=\{\sigma_{16}\eta^*_{23}\}$.
and $\pi^{16}_{31}\circ\eta_{31}\sigma_{32}=\{[\iota_{16},\eta_{16}\sigma_{17}]\}$. 
Hence, we have
$$
\sigma^*_{16}\eta_{38}\equiv \eta^*_{16}\sigma_{32}
\ \bmod\ \sigma_{16}\eta^*_{23},\ [\iota_{16},\eta_{16}\sigma_{17}].
$$
Since $\sigma^*=0$ and $\eta^*\sigma=\sigma\eta^*$ in the stable range, we get that 
$$
\eta^*_{16}\sigma_{32}=\sigma_{16}\eta^*_{23}
  +\sigma^*_{16}\eta_{38}+a[\iota_{16},\eta_{16}\sigma_{17}]
$$
for $a\in\{0,1\}$.
By (\ref{tn13sg}), (\ref{lmdsg}) and (\ref{smbsm}), we see that
\[\begin{split}
\nu_{13}\eta^*_{16}\sigma_{32}
&=\nu_{13}(\sigma_{16}\eta^*_{23}+\sigma^*_{16}\eta_{38}
  +a[\iota_{16},\eta_{16}\sigma_{17}])=\nu_{13}\sigma^*_{16}\eta_{38}\\
&\in\nu_{13}\circ\{\sigma_{16},2\sigma_{23},\sigma_{30}\}\circ\eta_{38}
 =\{\nu_{13},\sigma_{16},2\sigma_{23}\}\circ\eta_{31}\sigma_{32}\\
&=\xi_{13}\eta_{31}\sigma_{32}+\lambda\sigma_{31}\eta_{38}=
[\iota_{13},\sigma^2_{13}].
\end{split}\]
Thus, by the relation
$\nu_{13}\eta^*_{16}\equiv E\omega'\ \bmod\ \xi_{13}\eta_{31}$
\cite[Proposition 2.20(8)]{Og}, we have the equation (3) and completes the proof.
\end{proof}

Here, we need the following property of the generalized $P$-homomorphism \cite{T1, IM}.
\begin{lem}\label{Miy}
Let $k\ge 3$ and $X$, $Y$, $Z$ and $W$ be CW-complexes and 
$\alpha \in [E^{k}Y, E X\wedge X]$,
$\beta \in [Z, Y]$ and $\gamma \in [W,Z]$.
Suppose that the generalized P-homomorphisms 
$P:[E^{k}A, E X\wedge X] \to [E^{k-2}A, X]$ for $A=Y,$ $EW$ and $Y\cup_{\beta} CZ$
are well-defined and  $\alpha\circ E^{k}\beta=0$ and $\beta\circ\gamma=0$.
%P-homo 'ª'è‹`'³'ê'éðŒ X:S^{2n-1} 2-primary 'È'Ç'ª•K—v
Then the Toda bracket $\{P(\alpha), E^{k-2}\beta, E^{k-2}\gamma\}_{k-2}$ 
is well-defined and
\[
P \{\alpha, E^k\beta, E^k\gamma\}_{k}
\subset \{P(\alpha), E^{k-2}\beta, E^{k-2}\gamma\}_{k-2}.
\]
\end{lem}
\begin{proof}
By Proposition 2.5 of \cite{IM}, we have 
$P(\alpha)\circ E^{k-2}\beta= P(\alpha\circ E^{k}\beta)=0$.
Hence $\{P(\alpha), E^{k-2}\beta, E^{k-2}\gamma\}_{k-2}$ is well-defined.
We denote $Y\cup_{\beta} CZ$ by $C_{\beta}$ and the inclusion map $Y\to C_{\beta}$
 by $i_{\beta}$.
It is well-known that
$C_{E^{k}\beta} = E^{k}C_{\beta}$.
By Proposition 1.7 of \cite{T},
any element of $\{\alpha, E^{k}\beta, E^{k}\gamma \}_{k}$ is represented as
%denote ? represented as ? represented by ? 
$(-1)^{k}\bar{\alpha}\circ E^{k}\tilde{\gamma}$, where
$\bar{\alpha}\in[E^{k}C_{\beta}, E X\wedge X]$ is an extension of $\alpha$ and
$\tilde{\gamma}\in[E W, E^{k}C_{\beta}]$ is a coextension of $\gamma$.
By Proposition 2.5 of \cite{IM}, we have 
$P((-1)^{k}\bar{\alpha}\circ\Sigma^{k}\tilde{\gamma})=(-1)^{k}P(\bar{\alpha})\circ \Sigma^{k-2}\tilde{\gamma}$.
Since $P(\bar{\alpha})\circ \Sigma^{k-2}i_\beta
 = P(\bar{\alpha}\circ\Sigma^{k}i_\beta)=P(\alpha)$, the element
$P(\bar{\alpha})\in [\Sigma^{k-2}C_{\beta}, X]$ is an extension of $P(\alpha)$.
Hence we obtain 
\[
P((-1)^{k}\bar{\alpha}\circ\Sigma^{k}\tilde{\gamma})
  =(-1)^{k-2}P(\bar{\alpha})\circ\Sigma^{k-2}\tilde{\gamma}
    \in \{P(\alpha), \Sigma^{k-2}\beta, \Sigma^{k-2}\gamma\}_{k-2}
\]
 and
\[
P \{\alpha, \Sigma^k\beta, \Sigma^k\gamma\}_{k}
\subset \{P(\alpha), \Sigma^{k-2}\beta, \Sigma^{k-2}\gamma\}_{k-2}.
\]
\end{proof}

Now we show Lemma \ref{HPxl13}.

\begin{lem}\label{HPx13}
\begin{description}
\item[\rm (1)]
$H(P\xi_{13})\equiv\xi'\ \bmod\ 2\lambda', 2\xi'$\ and\  
$H(P\lambda)\equiv\lambda' \bmod 2\lambda',2\xi'$. 
\item[\rm (2)]
$H(P(\xi_{13}\eta_{31}))=\xi'\eta_{29} $\ and\  
$H(P(\lambda\eta_{31}))=\lambda'\eta_{29}$.
\end{description} 
\end{lem}
\begin{proof}
(2) is a direct consequence of (1). 
By the definition of $\xi_{12}$ and the fact that $\sigma_{12}\circ E\pi^{18}_{29}=\sigma_{12}\circ E^2\pi^{17}_{28}$, we know
$$
\xi_{12}\in\{\sigma_{12},\nu_{19},\sigma_{22}\}_1=
\{\sigma_{12},\nu_{19},\sigma_{22}\}_2.
$$
By Lemmas \ref{usgnep}(1), \ref{Miy}, \cite[Proposition 2.3]{T} and 
$HP\sigma_{13}=(H[\iota_6,\iota_6])\sigma_{11}=2\sigma_{11}$, we see that 
$$
P\xi_{13}\in P\{\sigma_{13},\nu_{20},\sigma_{23}\}_3\subset
\{P\sigma_{13}, \nu_{18},\sigma_{21}\}_1
$$
and 
$$
HP\xi_{13}\in 
\{2\sigma_{11}, \nu_{18},\sigma_{21}\}_1\equiv\xi'\ \bmod\ 2\lambda', 2\xi'.
$$
This leads to the first half of (1). 

By the fact that $\sigma_{12}\circ\pi^{19}_{31}=2\sigma_{11}\circ\pi^{18}_{30}=0$, 
we have 
$$
\{\sigma_{12},\nu_{19},\varepsilon_{22}\}=\{\sigma_{12},\nu_{19},\varepsilon_{22}\}_n\ (0\le n\le 12)
$$
and
$$
\{2\sigma_{11},\nu_{18},\varepsilon_{21}\}=\{2\sigma_{11},\nu_{18},\varepsilon_{21}\}_n\ (0\le n\le 11).
$$
So, by Lemmas \ref{omg'}(2) and \ref{Miy}, we obtain 
$$
P(E\omega')=P\{\sigma_{13},\nu_{20},\varepsilon_{23}\}_3\subset\{P(\sigma_{13}),\nu_{18},\varepsilon_{21}\}_1
$$
and
$$
HP(E\omega')\in H\{P(\sigma_{13}),\nu_{18},\varepsilon_{21}\}_1\subset\{2\sigma_{11},\nu_{18},\varepsilon_{21}\}_1=\lambda'\eta_{29}.
$$
By \cite[(3.3)]{MMO}, we see that 
$$
H(P\lambda)\equiv\pm\lambda'\ \bmod\ 2\lambda',2\xi',\eta_{11}\bar{\mu}_{12}.
$$
Hence, Lemma \ref{omg'} implies  
$$
HP(\lambda\eta_{31})=\lambda'\eta_{29}. 
$$
This and (1) lead to the second half of (1) and completes the proof. 
\end{proof}

\section{Proof of Theorem \ref{main}, I}

We recall from \cite[III-Proposition 2.6(5)]{Od} the relation 
$$
\bar{\nu}_7\omega_{15}\equiv 0\ \bmod\ \nu_7\sigma_{10}\kappa_{17}, \bar{\zeta}'_7
$$
and
\begin{equation}\label{bnu9omg}
\bar{\nu}_9\omega_{17}=0. 
\end{equation}

We show 
\begin{lem}\label{etnusg}
\begin{description}
\item[\rm (1)]
$\{\eta_{10},\nu_{11},\sigma_{14}\}=2[\iota_{10},\nu_{10}]$. 
\item[\rm (2)]
$\{\eta_n,\nu_{n+1},\sigma_{n+8}\}=0\ (n\ge 11)$. 
\end{description}
\end{lem}
\begin{proof}
First we show (2). 
From relations $\eta_n\zeta_{n+1}=0\ (n\ge 5)$ (\ref{et4zt}) and 
$\eta_{11}\circ[\iota_{12},\iota_{12}]=[\iota_{11},\eta^2_{11}]=0$ \cite[Theorem]{H}, 
it suffices to show that $\{\eta_{11},\nu_{12},\sigma_{15}\}=0$. Since 
$\eta_{11}\circ\pi^{12}_{23}+\pi^{11}_{16}\circ\sigma_{16}\!=\!\{\eta_{11}\zeta_{12},\eta_{11}\circ[\iota_{12},\iota_{12}]\}\!=\!0$, we have ${\rm Ind}\{\eta_{11},\nu_{12},\sigma_{15}\}\!=\!0$. 
From the fact that  $\{\eta_{10},\nu_{11},\sigma_{14}\}\!\subset\!\pi^{10}_{22}\!=\!\{[\iota_{10},\nu_{10}]\}$, we have 
$
\{\eta_{11},\nu_{12},\sigma_{15}\}\supset E\{\eta_{10},\nu_{11},\sigma_{14}\}=0\ \bmod\ 0$.

Next we show (1). %We use the relation $\nu_{10}\sigma_{13}=[\iota_{10},\eta_{10}]$ (\ref{nu10sg}).  
Let $\H P^2$ be the quaternionic projective plane and 
$i_\H: S^4\hookrightarrow\H P^2$ the inclusion. 
By the definition of the Toda bracket, there exist an extension $\bar{\eta}'_9\in[E^6\H P^2,S^9]$ of $\eta_9$ 
and a coextension $\tilde{\sigma}_{14}\in\pi_{22}(E^7\H P^2)$ of $\sigma_{14}$ such that $\{\eta_{10},\nu_{11},\sigma_{14}\}=E\bar{\eta}'_9\circ
\tilde{\sigma}_{14}$. 
By the Blakers-Massey theorem, we have
$\pi_{22}(E^7\H P^2,S^{11})\cong\pi_{22}(S^{15})$ and
$\pi_{23}(E^7\H P^2,S^{11})\cong\pi_{23}(S^{15})$. 
So, by the homotopy exact sequence of a pair $(E^7\H P^2,S^{11})$, we obtain 
$$
\pi_{22}(E^7\H P^2)=\{\tilde{\sigma}_{14},i\zeta_{11}\}\ (i=E^7i_\H: S^{11}\hookrightarrow E^7\H P^2).
$$
We consider the EHP sequence 
$$
\pi_{22}(E^7\H P^2)\rarrow{H}\pi_{22}(E(E^6\H P^2\wedge E^6\H P^2))\rarrow{P}\pi_{20}(E^6\H P^2). 
$$
We have 
$\pi_{21}(E^6\H P^2,S^{10})\cong\pi_{21}(S^{14})$ and 
$\partial\pi_{21}(E^6\H P^2,S^{10})=\{\nu_{10}\sigma_{13}\}$. So we obtain $(E^6i_\H)[\iota_{10},\eta_{10}]=0$. 
We have $\pi_{22}(E(E^6\H P^2\wedge E^6\H P^2))\cong\pi_{22}(S^{21})$ and $P(\pi_{22}(E(E^6\H P^2\wedge E^6\H P^2))=\{(E^6i_\H)[\iota_{10},\eta_{10}]\}=0$ (\ref{nu10sg}). Hence we get that
$$
H(\tilde{\sigma}_{14})=E(E^6i_\H\wedge E^6i_\H)\eta_{21}.
$$
Thus we see that 
\[
H(E\bar{\eta}'_9\circ\tilde{\sigma}_{14})
\!=\!E(\bar{\eta}'_9\wedge \bar{\eta}'_9)\circ H(\tilde{\sigma}_{14})
\!=\!E(\bar{\eta}'_9\wedge \bar{\eta}'_9)\circ E(E^6i_\H\wedge E^6i_\H)\eta_{21}
\!=\!\eta^3_{19}=4\nu_{19}.
\]
This leads to (1) and completes the proof. 
\end{proof}

We need relations
\begin{lem}\label{n15sgep}
$\{\nu_{13},\sigma_{16},\varepsilon_{23}\}_2=0$ \ and \  
$\{\nu_{n},\sigma_{n+3},\varepsilon_{n+10}\}=0$ for $n\ge 15$.
\end{lem}
\begin{proof}
By using $\pi^{14}_{30}=\{\omega_{14},\sigma_{14}\mu_{21}\}$,
$\pi^{k}_{k+16}=\{\omega_{k},\sigma_{k}\mu_{k+7}\}$ for $k\ge 18$,
$\pi^{\ell}_{\ell+11}=\{\zeta_{\ell}\}$ for $\ell\ge 13$,
(\ref{n11om}), (\ref{nu10sg}) and (\ref{ze6ep}), we have 
\[
{\rm Ind}\{\nu_{13},\sigma_{16},\varepsilon_{23}\}_2
 =\nu_{13}\circ E^2\pi^{14}_{30}+\pi^{13}_{24}\circ\varepsilon_{24}=0
\]
and
\[
{\rm Ind}\{\nu_{n},\sigma_{n+3},\varepsilon_{n+10}\}
 =\nu_{n}\circ\pi^{n+3}_{n+19}+\pi^{n}_{n+11}\circ\varepsilon_{n+11}=0
\]
for $n\ge 15$.
By Corollary \ref{sgnep1}, we have 
\[
0=E^2\{\nu_{11}, \sigma_{14}, \varepsilon_{21}\}
 \subset \{\nu_{13},\sigma_{16},\varepsilon_{23}\}_2 \ \bmod \ 0.
\]
This leads to the first half. Moreover, we have
\[\begin{split}
0=E^{n-13}\{\nu_{13},\sigma_{16},\varepsilon_{23}\}_2
 &\subset E^{n-15}\{\nu_{15},\sigma_{18},\varepsilon_{25}\}\\
 &\subset (-1)^{n-15}\{\nu_{n},\sigma_{n+3},\varepsilon_{n+10}\}
\ \bmod \ 0.
\end{split}\]
This leads to the second half and completes the proof. 
\end{proof}

We show 
\begin{lem}\label{bnusgep}
$\{\bar{\nu}_{11},\sigma_{19},\varepsilon_{26}\}\ni 0\ \bmod\ \sigma_{11}\nu_{18}\kappa_{21}=[\iota_{11},\kappa_{11}]$ and\\ $E\{\bar{\nu}_{11},\sigma_{19},\varepsilon_{26}\}=\{\bar{\nu}_{12},\sigma_{20},\varepsilon_{27}\}=0$.
\end{lem}
\begin{proof}
By \cite[Theorem 12.16]{T}, we have 
$$
{\rm Ind}\{\bar{\nu}_{11},\sigma_{19},\varepsilon_{26}\}=
\bar{\nu}_{11}\circ \pi^{19}_{35}+\pi^{11}_{27}\circ\varepsilon_{27}=
\bar{\nu}_{11}\circ\{\omega_{19},\sigma_{19}\mu_{26}\}+\{\sigma_{11}\mu_{18}\}\circ\varepsilon_{27}.  
$$ 
By (\ref{bnu6sg}), (\ref{bnu9omg}) and the fact that 
$\sigma_{11}\mu_{18}\varepsilon_{27}=\sigma_{11}\varepsilon_{18}\mu_{26}=0$
from \eqref{n5}, the indeterminacy is trivial.  
Similarly, we obtain
$$
{\rm Ind}\{\bar{\nu}_{12},\sigma_{20},\varepsilon_{27}\}=0.
$$ 
Therefore, the second assertion follows directly from the first assertion. 

By the relation $\{\nu_n,\eta_{n+3},\nu_{n+4}\}=\bar{\nu}_n$ for $n\ge 7$ \cite[Lemma 6.2]{T},
we have
$$
\{\bar{\nu}_{11},\sigma_{19},\varepsilon_{26}\}=\{\{\nu_{11},\eta_{14},\nu_{15}\},\sigma_{19},\varepsilon_{26}\}.
$$
By the Jacobi identity of Toda brackets and Lemmas \ref{etnusg} and \ref{n15sgep},
 we have
\[\begin{split}
\{\{\nu_{11},\eta_{14},\nu_{15}\},\sigma_{19},\varepsilon_{26}\}
&\equiv\{\nu_{11},\{\eta_{14},\nu_{15},\sigma_{18}\},\varepsilon_{26}\}
  +\{\nu_{11},\eta_{14},\{\nu_{15},\sigma_{18},\varepsilon_{25}\}\}\\
&\ni 0 \bmod\ \nu_{11}\circ\pi^{14}_{35}
+\pi^{11}_{27}\circ\varepsilon_{27}=\nu_{11}\circ\pi^{14}_{35}.
\end{split}\] 
By \cite[Theorem A]{Mi}, $\nu_{11}\eta_{14}=0$ and $\nu_{11}\sigma_{14}=0$, 
we have $\nu_{11}\circ\pi^{14}_{35}=\{E(\nu_{10}\lambda\nu_{31})\}$. 
Moreover, by the relations (\ref{sgm11n}) and (\ref{nulmn}), 
we have $E(\nu_{10}\lambda\nu_{31})=\sigma_{11}\nu_{18}\kappa_{17}
  =[\iota_{11},\kappa_{11}]$.
Thus, we have the first assertion. This completes the proof.
\end{proof}

Here we recall the definitions of $\delta_3$, $\bar{\sigma}'_6$ and  $\bar{\bar{\sigma}}'_6$ \cite[p. 13, 15]{MMO}: 
$$
\delta_3\in\{\varepsilon_3,\varepsilon_{11}+\bar{\nu}_{11},\sigma_{19}\}_1, 
$$
$$
\bar{\sigma}'_6\in\{\bar{\nu}_6,\varepsilon_{14}+\bar{\nu}_{14},\sigma_{22}\}_1,
$$
$$
\bar{\bar{\sigma}}'_6\in\{\nu_6,\eta_9,\bar{\sigma}_{10}\}_3.
$$

The indeterminacy of  $\{\varepsilon_3,\varepsilon_{11}+\bar{\nu}_{11},\sigma_{19}\}_1$ is 
$\varepsilon_3\circ E\pi^{10}_{26}+\pi^3_{20}\circ\sigma_{20}$. Since 
$\varepsilon_3\sigma_{11}=0$ \eqref{bnu6sg}, we have 
$\varepsilon_3\circ E\pi^{10}_{26}=0$. 
From the fact that 
$\pi^3_{20}=\{\bar{\varepsilon}', \bar{\mu}_3, \eta_3\mu_4\sigma_{13}\}$, 
$\bar{\varepsilon}'\sigma_{20}=0$ \cite[I-Proposition 3.1(5)]{Od} and 
\begin{equation}\label{mu3sg2} 
\mu_3\sigma^2_{12}=0 \text{\ \cite[(2.9)]{MMO}},
\end{equation} 
we obtain $\pi^3_{20}\circ\sigma_{20}=\{\bar{\mu}_3\sigma_{20}\}$. Hence, we have
$$
{\rm Ind}\{\varepsilon_3,\varepsilon_{11}+\bar{\nu}_{11},\sigma_{19}\}_1=\{\bar{\mu}_3\sigma_{20}\}.
$$

The indeterminacy of $\{\bar{\nu}_6,\varepsilon_{14}+\bar{\nu}_{14},\sigma_{22}\}_1$ is $\bar{\nu}_6\circ E\pi^{13}_{29}+\pi^6_{23}\circ\sigma_{23}$. 
By the fact that $\pi^{13}_{29}=\{\sigma_{13}\mu_{20}\}$ \cite[Theorem 12.16]{T} and
(\ref{bnu6sg}), we have
$\bar{\nu}_6\circ E\pi^{13}_{29}=0$. By the fact that 
$\pi^6_{23}=\{P(E\theta), \nu_6\kappa_9, \bar{\mu}_6, \eta_6\mu_7\sigma_{16}\}$
\cite[Theorem 12.7]{T}, \eqref{x13et} and \eqref{mu3sg2}, 
we obtain 
$\pi^6_{23}\circ\sigma_{23}=\{P(\xi_{13}\eta_{31})\}$. Hence, we have
$$
{\rm Ind}\{\bar{\nu}_6,\varepsilon_{14}+\bar{\nu}_{14},\sigma_{22}\}_1
 =\{P(\xi_{13}\eta_{31})\}.
$$

The indeterminacy of $\{\nu_6,\eta_9,\bar{\sigma}_{10}\}_3$ is $\nu_6\circ E^3\pi^6_{24}+\pi^6_{11}\circ\bar{\sigma}_{11}$. Since 
$[\iota_6,\iota_6]\circ\bar{\sigma}_{11}=0$ \cite[(5.7)]{MMO}, we have 
$\pi^6_{11}\circ\bar{\sigma}_{11}=0$. From the fact that 
$E^3\pi^6_{24}=\{\sigma_9\zeta_{16}, \eta_9\bar{\mu}_{10}\}$, we obtain 
$\nu_6\circ E^3\pi^9_{27}=0$. Hence, we have
$$
{\rm Ind}\{\nu_6,\eta_9,\bar{\sigma}_{10}\}_3=0.
$$

Since $\varepsilon_n+\bar{\nu}_n=\sigma_n\eta_{n+7}$ for $n\ge 10$ (\ref{et9sg}),
 and (\ref{bnu6sg}), 
we have 
$$
\{\varepsilon_3,\sigma_{11},\eta_{18}\sigma_{19}\}_1\subset\{\varepsilon_3,\varepsilon_{11}+\bar{\nu}_{11},\sigma_{19}\}_1
$$
and
$$
\{\bar{\nu}_6,\sigma_{14},\eta_{21}\sigma_{22}\}_1\subset\{\bar{\nu}_6,\varepsilon_{14}+\bar{\nu}_{14},\sigma_{22}\}_1.
$$
We change the definitions of $\delta_3$ and $\bar{\sigma}'_6$ as follows:
\begin{equation}\label{chd}
\delta'_3\in\{\varepsilon_3,\sigma_{11},\eta_{18}\sigma_{19}\}_1, 
\end{equation}
\begin{equation}\label{ch}
\bar{\sigma}''_6\in\{\bar{\nu}_6,\sigma_{14},\eta_{21}\sigma_{22}\}_1.
\end{equation}

By relations \cite[I-Proposition 3.1.(2)]{Od} and (\ref{et9sg}), we have
$$
\zeta'\varepsilon_{22}=\zeta'\bar{\nu}_{22}=0 \ \mbox{and so}, \ \zeta'\eta_{22}\sigma_{23}=0.
$$
By using \cite[Theorem 12.6 and 12.16]{T}, 
$\varepsilon_3\sigma_{11}\mu_{18}=0$ (\ref{bnu6sg}), 
$\mu_3\sigma_{12}\eta_{19}\sigma_{20}=\mu_3\eta_{12}\sigma_{13}^2=0$ (\ref{et9sg2}),
$\eta_3\bar{\varepsilon}_4\eta_{19}\sigma_{20}
 =\eta_3\varepsilon_4\bar{\nu}_{12}\sigma_{20}=0$ (\ref{epep}) and (\ref{bnu6sg}),
$\bar{\nu}_6\sigma_{14}\nu_{21}=0$ (\ref{bnu6sg}),
$\mu_6\sigma_{15}\eta_{22}\sigma_{23}=\mu_6\eta_{15}\sigma_{16}^2=0$ (\ref{et9sg2})
and
$\eta_6\bar{\varepsilon}_7\eta_{22}\sigma_{23}
 =\eta_6\bar{\varepsilon}_7\sigma_{22}\eta_{29}=0$ (\ref{bnu6sg}),
we have 
$$
{\rm Ind}\{\varepsilon_3,\sigma_{11},\eta_{18}\sigma_{19}\}_1= 
\varepsilon_3\circ\pi^{11}_{27}+\pi^3_{19}\circ\eta_{19}\sigma_{20}=0
$$ 
and
$$
{\rm Ind}\{\bar{\nu}_6,\sigma_{14},\eta_{21}\sigma_{22}\}_1= 
\bar{\nu}_6\circ E\pi^{13}_{29}+\pi^6_{22}\circ\eta_{22}\sigma_{23}=\{
\bar{\nu}_6\sigma_{14}\mu_{21},\zeta_{11}\eta_{22}\sigma_{23}\}=0.
$$
It is easy to check that this changing gives no influences  
to the computations in \cite{MMO}. 

By \cite[Proposition 2.6]{T} and (\ref{nu10sg}), we have
\[H\{\nu_{11},\sigma_{14},\eta_{21}\sigma_{22}\}_1
 =-P^{-1}(\nu_{10}\sigma_{13})\circ\eta_{22}\sigma_{23}=\eta^2_{21}\sigma_{23}.
\]
Hence, by the relations
\cite[(3.8)]{MMO} and $H(\xi')=\eta_{21}\sigma_{22}$ \cite[Lemma  12.19]{T},
 we take
$$
\{\nu_{11},\sigma_{14},\eta_{21}\sigma_{22}\}_1\ni\xi'\eta_{29}\ \bmod\ 
E\pi^{10}_{29}=\{\bar{\sigma}_{11},\bar{\zeta}_{11}\}.
$$
Since $\langle\nu,\sigma,\eta\sigma\rangle=\bar{\sigma}$ by 
\cite[I-Propsition 3.3 (4)]{Od} and \cite[(3.9).i)]{T}, we have
\begin{equation}\label{hbsgm''}
H(\bar{\sigma}''_6)=\bar{\sigma}_{11}+\xi'\eta_{29}.
\end{equation}

We set 
$$
\delta'_n=E^{n-3}\delta'_3\ (n\ge 3),\ \delta'=E^\infty\delta'_3
$$
and 
$$
\bar{\sigma}''_n=E^{n-6}\bar{\sigma}''_6\ (n\ge 6),\ \bar{\sigma}''=E^\infty\bar{\sigma}''_6.
$$

%So we have 
%$$ \bar{\bar{\sigma}}'_6\equiv\bar{\sigma}'_6\ \bmod\ E\pi^5_{29}=\{\delta_6,\bar{\mu}_6\sigma_{23},\nu_6\sigma_9\kappa_{16}\}$$and hence, by the relation $\nu_{11}\sigma_{14}=0$ (\ref{nu10sg}) and %$\pi^s_{24}=\{\delta,\bar{\mu}\sigma\}\cong(\Z_2)^2$, we see that 
%\begin{equation}\label{sgm"} \bar{\bar{\sigma}}'_{11}=\bar{\sigma}'_{11}. \end{equation}

Now we show 

\begin{lem}\label{bnu12}
$\bar{\sigma}''_{12}=\{\bar{\nu}_{12},\sigma_{20},\bar{\nu}_{27}\}=\{\nu_{12},\eta_{15},\bar{\sigma}_{16}\}=\bar{\bar{\sigma}}'_{12}$.
\end{lem}
\begin{proof}
Notice that the indeterminacies of the brackets 
$$
{\rm Ind}\{\bar{\nu}_{11},\sigma_{19},\bar{\nu}_{26}\}
={\rm Ind} \{\bar{\nu}_{11},\sigma_{19},\eta_{26}\sigma_{27}\}=0,
$$ 
respectively, because 
$
\bar{\nu}_{11}\circ\pi^{19}_{35}=\{\bar{\nu}_{11}\omega_{19},\bar{\nu}_{11}\sigma_{19}\mu_{26}\}=0$ by (\ref{bnu9omg}) and (\ref{bnu6sg}),
$
\pi^{11}_{27}\circ\bar{\nu}_{27}=\{\sigma_{11}\mu_{18}\bar{\nu}_{27}\}=0
$ by (\ref{mepbn2}) and 
$\pi^{11}_{27}\circ\eta_{27}\sigma_{28}
 =\{\mu_{11}\sigma_{20}\eta_{27}\sigma_{28}\}=0$ by (\ref{et9sg2}).
So, by (\ref{et9sg}), Lemma \ref{bnusgep} and (\ref{ch}), we obtain 
$$
\{\bar{\nu}_{11},\sigma_{19},\bar{\nu}_{26}\}
\!=\!\{\bar{\nu}_{11},\sigma_{19},\varepsilon_{26}\}
 + \{\bar{\nu}_{11},\sigma_{19},\eta_{26}\sigma_{27}\}
\!=\!a[\iota_{11},\kappa_{11}]+\bar{\sigma}''_{11}
$$
for $a=0$ or $a=1$.
%By \cite[(5.39)]{MMO} and the fact that $E^\infty: \pi^{19}_{43}\to\pi^s_{24}$ is an isomorphism, we have \begin{equation}\label{whnu219}[\iota_{19},\nu^2_{19}]=\bar{\sigma}'_{19}.\end{equation}
Hence, we obtain 
$$
\bar{\sigma}''_{12}\in E\{\bar{\nu}_{11},\sigma_{19},\bar{\nu}_{26}\}
\subset\{\bar{\nu}_{12},\sigma_{20},\bar{\nu}_{27}\}\ \bmod\ \bar{\nu}_{12}\circ\pi^{20}_{36}+\pi^{12}_{28}\circ\bar{\nu}_{28}=0.
$$

By the Jacobi identity of Toda brackets, \cite[Lemma 6.2]{T}, Lemma \ref{etnusg}.(2)
and \cite[I-Proposition 3.4(3)]{Od},
 we obtain 
\[\begin{split}
\bar{\sigma}''_{12}&=\{\bar{\nu}_{12},\sigma_{20},\bar{\nu}_{27}\}
  =\{\{\nu_{12},\eta_{15},\nu_{16}\},\sigma_{20},\bar{\nu}_{27}\}\\
&\equiv\{\nu_{12},\{\eta_{15},\nu_{16},\sigma_{19}\},\bar{\nu}_{27}\}
  +\{\nu_{12},\eta_{15},\{\nu_{16},\sigma_{19},\bar{\nu}_{26}\}\}\\
&\equiv\{\nu_{12},\eta_{15},\bar{\sigma}_{16}\}\ \bmod\ \nu_{12}\circ\pi^{15}_{36}.
\end{split}\]
Since $\pi^{15}_{36}=\{\eta_{15}\bar{\kappa}_{16},\sigma^3_{15}, (E^2\lambda)\nu_{33}\}\cong(\Z_2)^3$ \cite[Theorem A]{Mi}, (\ref{nu10sg}) and (\ref{nulmn}), we have 
$\nu_{12}\circ\pi^{15}_{36}=\{\nu_{12}(E^2\lambda)\nu_{33}\}=0$. 
This %, (\ref{sgm"}) and (\ref{whnu219})  
implies
$$
\{\nu_{12},\eta_{15},\bar{\sigma}_{16}\}=\bar{\sigma}''_{12}.
$$
and completes the proof.
\end{proof}

We show
\begin{lem}\label{dl'dl}
$\delta'_3\equiv\delta_3\ \bmod\ \bar{\mu}_3\sigma_{20}$, $\bar{\sigma}''_6\equiv\bar{\sigma}'_6\ \bmod\ 
P(\xi_{13})\eta_{29}$
\ and \  
$\bar{\sigma}''_6\equiv\bar{\bar{\sigma}}'_6+P(\xi_{13})\eta_{29}\ \bmod\ \nu_6\sigma_9\kappa_{16}$. 
\end{lem}
\begin{proof}
By the fact that $\pi^3_{20}=\{\bar{\mu}_3,\eta_3\mu_4\sigma_{13},\bar{\varepsilon}'\}$, $\mu_3\sigma^2_{12}=0$ (\ref{mu3sg2}) and $\bar{\varepsilon}'\sigma_{20}=0$ \cite[I-Proposition 3.1(5)]{Od}, we get the first half.  

We observe 
$$
\bar{\sigma}''_6\in\{\bar{\nu}_6,\sigma_{14},\eta_{21}\sigma_{22}\}_1\subset\{\bar{\nu}_6,\varepsilon_{14}+\bar{\nu}_{14},\sigma_{22}\}_1\ \bmod\ \bar{\nu}_6\circ E\pi^{13}_{29}+\pi^6_{23}\circ\sigma_{23}.
$$
We know $\pi^6_{23}=\{P(E\theta),\nu_6\kappa_9,\bar{\mu}_6,\eta_6\mu_7\sigma_{16}\}$, 
$\nu_6\kappa_9\sigma_{23}=0$ (\ref{k7s}) and 
$P(E\theta)\sigma_{23}=P(\xi_{13}\eta_{31})$ (\ref{x13et}). 
This leads to the second. 

By the relations $H(\bar{\bar{\sigma}}'_6)=\bar{\sigma}_{11}$ \cite[(3.8)]{MMO}, 
(\ref{hbsgm''}) and Lemma \ref{HPx13}(2), we have
$$
H(\bar{\sigma}''_6)=H(\bar{\bar{\sigma}}'_6+P(\xi_{13})\eta_{29})
$$
and
$$
\bar{\sigma}''_6\equiv\bar{\bar{\sigma}}'_6+P(\xi_{13})\eta_{29}\ \bmod\ E\pi^5_{29}=\{\delta_6,\bar{\mu}_6\sigma_{23},\nu_6\sigma_9\kappa_{16}\}.
$$
Hence, by the fact that $\bar{\sigma}''_{12}=\bar{\bar{\sigma}}'_{12}$ (Lemma \ref{bnu12}),  we obtain the second half.
This completes the proof. 
\end{proof}

\section{Proof of Theorem \ref{main}, II: Some relations in homotopy groups of the rotation group}
Let $SO(n)$ be the $n$-th rotation group, $i_{k,n}: SO(k)\hookrightarrow SO(n)\ (k\leq n)$ the inclusion, $i_n=i_{n,n+1}$ and $p_n: SO(n)\to SO(n)/SO(n-1)=S^{n-1}$ the projection. Let $\Delta: \pi_k(S^n)\to\pi_{k-1}(SO(n))$ be the connecting map associated with the bundle $p_{n+1}: SO(n+1)\to S^n$.
Suppose that there exist elements $\alpha\in\pi_k(S^n)$ and $\beta\in\pi_k(SO(n+1))$ satisfying the relation 
$$
{p_{n+1}}_*\beta=\alpha.
$$
Then, $\beta$ is called a lift of $\alpha$ and is written $[\alpha]$. 
For a lift $[\alpha]\in\pi_k(SO(m))$, we write $[\alpha]_n={i_{m,n}}_*[\alpha]\in\pi_k(SO(n))$ for $m\leq n$. 

Let $J: \pi_k(SO(n))\to\pi_{k+n}(S^n)$ be the $J$-homomorphism. Denote by $R^n_k$ the $2$-primary component of $\pi_k(SO(n))$. 
We use the exact sequence induced from the fibration $p_n$: 
$$
({\mathcal{R}}^n_k) \hspace{0.5cm}
\cdots\rarrow{}\pi^{n}_{k+1}\rarrow{\Delta}R^n_k\rarrow{{i_{n+1}}_*} R^{n+1}_k\rarrow{{p_{n+1}}_*}\pi^n_k\rarrow{}\cdots. 
$$

As it is well-known \cite[p. 233-4]{WG},
$$
J(\Delta\alpha)=\pm[\iota_n,\alpha] \ (\  \alpha\in\pi_k(S^n)).
$$
We also know
\begin{equation}\label{Js}
J([\alpha]\beta)=J[\alpha]\circ E^n\beta\ (\alpha\in\pi_k(S^n), \beta\in\pi_l(S^k)).
\end{equation}

Although the following is well-known, we show 
\begin{lem}\label{Ji0}
Assume that elements $\alpha \in\pi_{h}(SO(n))$,
$\beta\in \pi_{l}(S^h)$ and $\gamma\in \pi_{k-1}(S^l)$ satisfy the conditions
$\alpha\beta=0$ and $\beta\gamma=0$. Then the Toda bracket $\{J\alpha,E^n\beta,E^n\gamma\}_n$ is well-defined and 
$$J\{\alpha,\beta,\gamma\} \subset (-1)^n\{J\alpha,E^n\beta,E^n\gamma\}_n.$$
\end{lem}
\begin{proof}
We recall \cite[(11.2)]{T} that the $J$-homomorphism is defined by the composition
$$J={G_n}_*\circ E^n : \pi_{k}(SO(n)) \to \pi_{k+n}(E^nSO(n)) \to \pi_{k+n}(S^n).$$
Here $G_n:E^nSO(n)\to S^n$ is the Hopf construction obtained from the action of $SO(n)$ as the rotations of $S^{n-1}$. 
Since $J\alpha\circ E^n\beta={G_n}_*\circ E^n\alpha\circ E^n\beta=0$,
the Toda bracket $\{J\alpha,E^n\beta,E^n\gamma\}_n$ is well-defined.
By \cite[Propositions 1.2--3]{T}, we have
\begin{eqnarray*}
J\{\alpha,\beta,\gamma\} &=& {G_n}_*\circ E^n\{\alpha,\beta,\gamma\}\\
&\subset& {G_n}_*\circ (-1)^n \{E^n\alpha,E^n\beta,E^n\gamma\}_{n}\\
&\subset& (-1)^n\{{G_n}_*\circ E^n\alpha,E^n\beta,E^n\gamma\}_{n}\\
&=& (-1)^n\{J\alpha,E^n\beta,E^n\gamma\}_{n}.
\end{eqnarray*}
\end{proof}

By (\ref{Ji0}), we have
\begin{equation}\label{Ji}
J({i_{m,n}}_*\gamma)=(-1)^{n-m}E^{n-m}J\gamma\ (\gamma\in\pi_k(SO(m)).
\end{equation}

Since $\Delta\iota_7\in\pi_6(SO(7))=0$ \cite[Table, p. 162]{Ke}, 
there exists a lift $[\iota_7]\in R^8_7$ of $\iota_7$. By \cite[(4.1)]{HKM}, we have
\begin{equation}\label{Dsm9}
\Delta\sigma_9=[\iota_7]_9(\bar{\nu}_7+\varepsilon_7).
\end{equation}
By (\ref{bnu6sg}), we obtain $\Delta(\sigma^2_9)=[\iota_7]_9(\bar{\nu}_7+\varepsilon_7)\sigma_{15}=0$ 
and so, there exists a lift $[\sigma^2_9]\in R^{10}_{23}$\ of $\sigma^2_9$. 

We show 
\begin{lem}\label{id}
$\{p_{n+1},i_{n+1},\Delta\iota_n\}\ni \iota_n\ \bmod\ 2\iota_n$.
\end{lem}
\begin{proof}
Let $P^n$ be the real $n$ dimensional projective space, $\gamma_n: S^n\to P^n$ the projection. Let $j_n: \P^{n-1}\to SO(n)$ be the canonical inclusion. We know 
$$
\Delta\iota_n=j_n\gamma_{n-1},\ i_n\circ j_n=j_{n+1}\circ i'_n,\ p_n\circ j_n=p'_{n-1},
$$
where $i'_n: \P^{n-1}\to\P^n$ is the inclusion and $p'_n: \P^n\to S^n$ the collapsing map. 
We have 
\begin{eqnarray*}
\{p_{n+1},i_{n+1},\Delta\iota_n\}
&=&\{p_{n+1},i_{n+1},j_n\gamma_{n-1}\}\\
&\subset&\{p_{n+1},i_n\circ j_n,\gamma_{n-1}\}\\
&=&\{p_{n+1},j_{n+1}\circ i'_n,\gamma_{n-1}\}\\
&\supset&\{p_{n+1}\circ j_{n+1},i'_n,\gamma_{n-1}\}\\
&=&\{p'_n,i'_n,\gamma_{n-1}\}\\
&\ni&\iota_n\ \bmod\ p_{n+1}\circ\pi_n(SO(n+1))+[E\P^{n-1},S^n]\circ E\gamma_{n-1}.
\end{eqnarray*}
Since ${p_{n+1}}_* \pi_n(SO(n+1))=\{{1+(-1)^{n-1}}\iota_n\}$ and $[E\P^{n-1},S^n]\circ E\gamma_{n-1}=\{E(p'_{n-1}\circ\gamma_{n-1})\}=
\{{1+(-1)^n}\iota_n\}$, we obtain 
$$
p_{n+1}\circ\pi_n(SO(n+1))+[E\P^{n-1},S^n]\circ E\gamma_{n-1}=\{2\iota_n\}. 
$$
This completes the proof. 
\end{proof}

Let us recall the element $\psi_{10}\in\pi^{10}_{33}$ \cite[(4.27)]{MMO}. We show
\begin{lem}\label{Jim}
$J[\iota_7]=\sigma_8$ \ and \ $J[\sigma^2_9]= \psi_{10}$. 
\end{lem}
\begin{proof}
The first is just \cite[(2.2)]{HKM}. 
Since $[\iota_7]_{10}(\bar{\nu}_7+\varepsilon_7)=i_{10}\circ\Delta\sigma_9=0$ (\ref{Dsm9}), we can define the Toda bracket
$$
\{[\iota_7]_{10},\bar{\nu}_7+\varepsilon_7, \sigma_{15}\}\subset R^{10}_{23}.
$$
By Lemma \ref{id} and (\ref{Dsm9}), we have 
$$
p_{10}\circ\{[\iota_7]_{10},\bar{\nu}_7+\varepsilon_7, \sigma_{15}\}
=-\{p_{10}, [\iota_7]_{10},\bar{\nu}_7+\varepsilon_7\}\circ\sigma_{16}
\supset-\{p_{10},i_{10},\Delta\iota_9\circ\sigma_8\}\circ\sigma_{16}
$$
$$
\supset-\{p_{10}, i_{10}, \Delta\iota_9\}\circ\sigma^2_9\ni \sigma^2_9 \bmod\ {p_{10}}_*R^{10}_{16}\circ\sigma_{16}=\{2\sigma^2_9\}.
$$
Hence, we can take 
$$
[\sigma^2_9]\in\{[\iota_7]_{10},\bar{\nu}_7+\varepsilon_7, \sigma_{15}\}.
$$
By (\ref{Ji}), we have
\begin{eqnarray*}
J[\sigma^2_9]&\in& J\{[\iota_7]_{10},\bar{\nu}_7+\varepsilon_7, \sigma_{15}\}\\
& \subset& \{J[\iota_7]_{10},\bar{\nu}_{17}+\varepsilon_{17}, \sigma_{25}\}_{10}\\
& \subset& \{\sigma_{10},\bar{\nu}_{17}+\varepsilon_{17}, \sigma_{25}\}_{4}.
\end{eqnarray*}
Therefore, by the definition of $\psi_{10}$,  we have the second. 
This completes the proof. 
\end{proof}

%We know the following \cite{Ka, Mi1}: $\Delta\nu_5=0,\ \Delta\nu_{13}=[\nu_5]_{13}\sigma_8=[\bar{\nu}_6+\varepsilon_6]_{13}\eta_{14}\ne 0$. 
By use of the result in \cite[Table, p. 161]{Ke} and $({\mathcal{R}}^{8n+5}_{8n+7})$, we obtain 
$$
\Delta\nu_{8n+5}\ne 0\ (n\ge 1). 
$$
In particular, we show 
\begin{lem}\label{wnu21}
$\Delta\nu_{21}=[\sigma^2_9]_{21}$ \ and \  $[\iota_{21},\nu_{21}]=\psi_{21}$.
\end{lem}
\begin{proof}
Since $J\Delta\nu_{21}=-[\iota_{21},\nu_{21}]$ and 
$J[\sigma^2_9]_{21}=J({i_{10,21}}_*[\sigma^2_9])=-E^{11}J[\sigma^2_9]=-\psi_{21}$,
the second leads to the first. 
By \cite[Table, p. 161]{Ke} and \cite{Bo}, we have
$$
R^{21}_{23}\cong\Z\oplus\Z_2,\  R^{22}_{23}\cong\Z.
$$
By \cite[(4.37)]{MMO} and Lemma \ref{Jim}, we observe that
$$
J{[\sigma^2_9]}_{21}=\psi_{21}\ne 0.
$$
Hence, the direct summand $\Z_2$ in $R^{21}_{23}$ is generated by $[\sigma^2_9]_{21}$.  
We consider the exact sequence $({\mathcal{R}}^{21}_{23})$: 
$$
\pi^{21}_{24}\rarrow{\Delta}R^{21}_{23}\rarrow{{i_{21}}_*} R^{22}_{23}\rarrow{{p_{22}}_*}\pi^{21}_{23}. 
$$
Since 
$[\iota_{21},\eta^2_{21}]=4\sigma^*_{21}\ne 0$ \cite[Lemma 8.3]{Mi}, ${p_{22}}_*$ is trivial and 
${i_{21}}_*$ is a split epimorphism. Hence, we obtain the relation. 
\end{proof}

The second result in Lemma \ref{wnu21}
is an excluded case in \cite[Theorem 3.6(9)]{Th}. This theorem ensures the following.
\begin{conj}\label{conj1}
There exists a lift $[\sigma^2_{16k-7}]\in R^{16k-6}_{16k+7}$ of $\sigma^2_{16k-7}$ such that $\Delta(\nu_{16k+5})=[\sigma^2_{16k-7}]_{11}$ for $k\ge 2$.
\end{conj}

We show 
\begin{lem}\label{n1}
\begin{description}
\item[\rm (1)]
$\nu^2_5\bar{\sigma}_{11}=\eta_5\bar{\sigma}'_6$.
\item[\rm (2)]
$2(\sigma_9\nu^*_{16})\equiv\ \nu^2_9\bar{\sigma}_{15}\ \bmod\ 4\sigma_9\xi_{16}$, \ 
$4(\xi_{12}\sigma_{30})=\nu^2_{12}\bar{\sigma}_{18}$ \ 
and \\ 
$\nu^2_{17}\bar{\sigma}_{23}=4(\xi_{17}\sigma_{35})
=[\iota_{17},\eta^2_{17}\sigma_{19}]\ne 0$.
\end{description}
\end{lem}
\begin{proof}
We recall from \cite[II-Proposition 2.1(7)]{Od} the relation
$$
\eta_5\bar{\sigma}'_6\equiv \nu^2_5\bar{\sigma}_{11}\ \bmod\ \eta_5\bar{\mu}_6\sigma_{23}.
$$
In the stable range, $\nu^2\bar{\sigma}=\eta\circ\langle\nu,\eta,\bar{\sigma}\rangle=\eta\bar{\sigma}'$ and $\eta\bar{\mu}\sigma\ne 0$. This leads to (1). 

We recall from \cite[p. 70 and I-Proposition 5.1(1)]{Od} the relations
$$
2\phi'''\equiv \nu^2_7\bar{\sigma}_{13}\ \bmod\ 2E^2\phi'',
$$
$$
2E^2\phi''\equiv 2(\sigma'\xi_{14})\ \bmod\ 2(\sigma'(E\lambda+2\xi_{14}))
$$
and 
$$
2E^4\phi''\equiv 4(\sigma_9\xi_{16})\ \bmod\ 4(\sigma_9(E^3\lambda))+8(\sigma_9\xi_{16}).
$$
We know $8(\sigma_9\xi_{16})=0$ and $\sigma_9(E^3\lambda)=2(\sigma_9\nu^*_{16})$ \cite[I-Proposition 3.1(8)]{Od},
$2(\sigma_9\nu^*_{16})=2E^2\phi'''$ and $4(\sigma_9\nu^*_{16})=0$
\cite[I-Proposition 5.1(2);(3)]{Od}.
This leads to the first of (2).

By using \cite[Corollary 12.25]{T} and \cite[I-Theorem 1(a), Propositions 3.5(3), 6.3(8);(11)]{Od}, we have
$$
2(\xi_{12}\sigma_{30})=\sigma_{12}\xi_{19},\ [\iota_{12},\sigma^2_{12}]=\sigma_{12}(\xi_{19}+\nu^*_{19})\ \ \text{and}\ \ %
[\iota_{17},\eta^2_{17}\sigma_{19}]=4(\xi_{17}\sigma_{35}).
$$
This leads to the rest of (2) and completes the proof. 
\end{proof}

Denote by $M^n=S^{n-1}\cup_{2\iota_{n-1}}e^n$ be $\Z_2$-Moore space and by $i_n: S^{n-1}\to M^n$ the inclusion map. Let $\bar{\eta}_n\in[M^{n+2},S^n],\ \tilde{\eta}_n\in\pi_{n+2}(M^{n+1})\ (n\ge 3)$ be an extension and a coextension of $\eta_n$, respectively. 
Notice that $\tilde{\eta}_n\in\pi_{n+2}(M^{n+1})\cong\Z_4$ is a generator. %satisfying ${p_{n+1}}_*\tilde{\eta}_n=\eta_{n+1}$. 
We note that $2\nu_n=\pm\bar{\eta}_n\tilde{\eta}_{n+1}$ for $n\ge 5$. 
We show 
\begin{lem}\label{dl''}
$\delta\equiv\eta\eta^*\sigma\ \bmod\ \bar{\mu}\sigma$ \ and \ $\delta'=\eta\eta^*\sigma$.  
\end{lem}
\begin{proof}
By \cite[I-Propositions 3.4(7), 3.5(9)]{Od}, we have 
$$
\eta_8\sigma_9\eta^*_{16}\equiv\delta_8\ \bmod\ \bar{\mu}_8\sigma_{25},\ \eta_8\sigma^2_9\mu_{23}.
$$
We know $\eta_9\sigma^2_{10}\mu_{24}=0$ (\ref{et9sg2}).
This leads to the relation
\begin{equation}\label{etsget*}
\eta_9\sigma_{10}\eta^*_{17}\equiv\delta_9\ \bmod\ \bar{\mu}_9\sigma_{26}
\end{equation}
and the first half.

By the relations $\eta\sigma=\bar{\nu}+\varepsilon$ and $\varepsilon\omega=\eta\eta^*\sigma$ \cite[(6.3)]{MMO}, we have 
$$
\langle\varepsilon,\sigma,\eta\sigma\rangle\subset\langle\varepsilon,\sigma,\bar{\nu}\rangle+\langle\varepsilon,\sigma,\varepsilon\rangle.
$$
By \cite[(6.1)]{MMO}, we have
$$
\langle\varepsilon,\sigma,\bar{\nu}\rangle=\{\eta\eta^*\sigma\}.
$$

We recall the relation \cite[Lemma 4.1]{GM}:
$$
\langle\eta\bar{\eta},\tilde{\eta},\nu\rangle=\varepsilon.
$$
By the Jacobi identity, we have
\[\begin{split}
\langle\varepsilon,\sigma,\varepsilon\rangle&
=\langle\langle\eta\bar{\eta},\tilde{\eta},\nu\rangle,\sigma,\varepsilon\rangle\\
&\equiv\langle\eta\bar{\eta},\langle\tilde{\eta},\nu,\sigma\rangle,\varepsilon\rangle
+\langle\eta\bar{\eta},\tilde{\eta},\langle\nu,\sigma,\varepsilon\rangle\rangle\ \bmod\ \{\varepsilon\omega\}+\eta\bar{\eta}\circ\pi^s_{23}(M^2).
\end{split}\]
We know $\langle\tilde{\eta},\nu,\sigma\rangle\subset\pi^s_{14}(M^2)=0$\ and $\langle\nu,\sigma,\varepsilon\rangle=0$ (Lemma \ref{sgnep}). This implies
$$
\langle\varepsilon,\sigma,\varepsilon\rangle\ni 0\ \bmod\ \{\varepsilon\omega\}+\eta\bar{\eta}\circ\pi^s_{23}(M^2).
$$
Let $\widetilde{\sigma^2}\in\pi^s_{16}(M^2)$ be a coextension of $\sigma^2$.
Then, by the definition of $\eta^*$ \cite[p. 153]{T} and \cite[(3.9).i)]{T}, we have
$$
\bar{\eta}\widetilde{\sigma^2}\in\langle\eta,2\iota,\sigma^2\rangle\ni\eta^*\ \bmod\ \eta\rho=\mu\sigma.
$$
By easy calculations making use of the cofibration $S^1 \to M^2 \to S^2$, we obtain
$$
\pi^s_{23}(M^2)=\{\tilde{\eta}\bar{\kappa},\widetilde{\sigma^2}\sigma,i\nu\bar{\sigma}\}
 \cong\Z_4\oplus(\Z_2)^2
\ \text{ and }\ 2\tilde{\eta}\bar{\kappa}=i\eta^2\bar{\kappa} \ (i=E^{\infty}i_{2}). 
$$
By the relations $2\nu=\pm\bar{\eta}\tilde{\eta}$, $\eta\sigma^2=0$ and 
$\bar{\eta}i\nu=\eta\nu=0$, we have
$$
\eta\bar{\eta}\circ\pi^s_{23}(M^2)=\eta\circ\{2\nu\bar{\kappa},\eta^*\sigma\}=\{\eta\eta^*\sigma\}.
$$
This implies 
$$
\delta'\equiv 0\ \bmod\ \eta\eta^*\sigma.
$$
Hence, the second half follows from the first half and  Lemma \ref{dl'dl}. This completes the proof. 
\end{proof}

Now we show 
\begin{thm}\label{bsg'dt}
$\delta'_9=\eta_9\sigma_{10}\eta^*_{17}$ \ and \ $\bar{\sigma}''_{19}+\delta'_{19}=[\iota_{19},\nu^2_{19}]$.
\end{thm}
\begin{proof}
 By the similar proof to \cite[II-Proposition 2.1(10)]{Od}, we obtain 
$$
\eta_9\psi_{10}\equiv\ \bar{\sigma}''_9+\delta'_9\ \bmod\ \bar{\mu}_9\sigma_{26}, \sigma^2_9\eta_{23}\mu_{24}.
$$
Since $\eta_{20}\psi_{21}=[\eta_{20},\eta_{20}\nu_{21}]=0$ by Lemma \ref{wnu21}, 
we obtain 
\begin{equation}\label{bsg'20} 
\bar{\sigma}''_{20}\equiv \delta'_{20}\ \bmod\ \bar{\mu}_{20}\sigma_{37}.
\end{equation}

By the fact that $\pi^{20}_{44}\cong\pi^s_{24}$ and Lemma \ref{dl''},
we have 
\begin{equation}\label{det20}
\eta_{20}\eta^*_{21}\sigma_{37}=\eta_{20}\sigma_{21}\eta^*_{28}\ \text{ and }\ %
\delta'_{20}=\eta_{20}\eta^*_{21}\sigma_{37}.
\end{equation}
From (\ref{bsg'20}) and (\ref{det20}), we have
$
0=\bar{\sigma}''_{20}+\eta_{20}\eta^*_{21}\sigma_{37}+a\bar{\mu}_{20}\sigma_{37}
$
for $a=0$ or $a=1$, and 
$
\bar{\sigma}''=\eta\eta^*\sigma+a\bar{\mu}\sigma.
$
By Lemma \ref{n1}(2), we have
$
\nu^2\bar{\sigma}=0.
$
Hence, by \cite[Lemma 5.12 and 14.1.i)]{T} and Lemma \ref{bnu12}, we have 
$$
0=\nu^2\bar{\sigma}=\langle\eta,\nu,\eta\rangle\circ\bar{\sigma}=\eta\circ\langle\nu,\eta,\bar{\sigma}\rangle=\eta(\eta\eta^*\sigma+a\bar{\mu}\sigma)=4\nu^*\sigma+a\eta\bar{\mu}\sigma=a\eta\bar{\mu}\sigma.
$$
This induces $a=0$ and the second relation. 

The third one follows from \cite[(5.38-40)]{MMO} and Lemma \ref{dl'dl}. 

By Lemma \ref{dl'dl} and (\ref{etsget*}), we have
$$
\delta'_9\equiv\eta_9\sigma_{10}\eta^*_{17}\ \bmod\ \bar{\mu}_9\sigma_{16}.
$$
So, in the stable range, we have
$$
\delta'\equiv\eta\sigma\eta^*\ \bmod\ \bar{\mu}\sigma.
$$
This and Lemma \ref{dl''} lead to the first and completes the proof. 
\end{proof}

Since $\varepsilon_9\circ E^4\pi^{13}_{29}=\{\varepsilon_9\sigma_{17}\mu_{24}\}=0$ and 
$$
\pi^9_{25}\circ\eta_{25}\sigma_{26}=\{\sigma_9\nu^3_{16}, \sigma_9\mu_{16}, \sigma_9\eta_{15}\varepsilon_{16}, \mu_9\sigma_{18}\}\circ\eta_{25}\sigma_{26}=0,
$$ 
we obtain ${\rm Ind}\{\varepsilon_9,\sigma_{17},\eta_{24}\sigma_{25}\}_4
=0$. Lemmas \ref{bnu12}, \ref{dl''} and Theorem \ref{bsg'dt} imply Theorem \ref{main}. 

By \cite[Proposition]{Ba}, \cite[Example 2.3(3)]{MM} and Lemma \ref{wnu21},
$$
\sigma_{21}\omega_{28}-\omega_{21}\sigma_{37}=[\iota_{21},\nu_{21}]=\eta_{21}\sigma^*_{22}=\psi_{21}.
$$

We recall from \cite[p. 153]{T} that the definition of ${\eta^*}'\in\pi^{15}_{31}$ is 
\begin{equation}\label{et*'}
{\eta^*}'\in\{\sigma_{15},4\sigma_{22},\eta_{29}\}_1.
\end{equation}
We show 
\begin{lem}\label{etsg*}
$E{\eta^*}'\circ\sigma_{32}=[\iota_{16},\eta_{16}\sigma_{17}]$ \ and \ 
$\eta_{16}\sigma^*_{17}=\sigma_{16}\omega_{23}+\omega_{16}\sigma_{32}+[\iota_{16},\eta_{16}\sigma_{17}]$.
\end{lem}

\begin{proof}
By (\ref{et*'}), $2\sigma^2_{16}=0$, 
$\pi^{16}_{31}=\{\rho_{16},\bar{\varepsilon}_{16},[\iota_{16},\iota_{16}]\}$,
and $\rho_{16}\eta_{31}=\sigma_{16}\mu_{23}$ \cite[Proposition 12.20.i)]{T}, we
obtain 
$$
E{\eta^*}'\in\{\sigma_{16},4\sigma_{23},\eta_{30}\}\supset\{0,2\iota_{30},\eta_{30}\}=\pi^{16}_{31}\circ\eta_{31}=\{\sigma_{16}\mu_{23},[\iota_{16},\eta_{16}]\}.
$$
Since $\sigma_{16}\mu_{23}\sigma_{32}=\sigma_{16}^2\mu_{30}=0$ from \cite[(2.3)]{MMO},
 we have 
$$
E{\eta^*}'\circ\sigma_{32}\in\{[\iota_{16},\eta_{16}\sigma_{17}]\}.
$$
This and the relation $[\iota_{16},\eta_{16}]\equiv E{\eta^*}'\ \bmod\ E^2\pi^{14}_{30}$ \cite[p. 160]{T} lead to the first half. 

By the relation 
$\eta_{15}\sigma^*_{16}
 ={\eta^*}'\sigma_{31}+\sigma_{15}\omega_{22}+\omega_{15}\sigma_{31}$ 
\cite[III-Proposition 2.5(5)]{Od}
and the first half leads to the second half and the proof is complete.
\end{proof}

Finally we show 
\begin{prop}
$\eta^*_{16}\bar{\nu}_{32}\equiv \nu^*_{16}\nu^2_{34}\ \bmod\ [\iota_{16},\nu^3_{16}]$ \ and \ $\nu^*_{19}\nu^2_{37}=\eta^*_{19}\bar{\nu}_{35}=\omega_{19}\bar{\nu}_{35}=[\iota_{19},\nu^2_{19}]$.
\end{prop}
\begin{proof}
The last equality is just \cite[(5.34)]{MMO}.

We have $\eta^*_{16}\nu_{32}=0$, because 
we see that  
$$
\eta^*_{16}\nu_{32}\in\{\sigma_{16},2\sigma_{23},\eta_{30}\}\circ\nu_{32}=-\sigma_{16}\circ\{2\sigma_{23},\eta_{30},\nu_{31}\}
$$
$$
\supset \sigma^2_{16}\circ\{2\iota_{30},\eta_{30},\nu_{31}\}=0\ \bmod\ \sigma_{16}\circ\pi^{23}_{32}\circ\nu_{32}=0.
$$
By using \cite[Lemma 6.2]{T}, 
$\pi^{16}_{37}=\{\eta_{16}\bar{\kappa}_{17},\sigma^3_{16},
(E^3\lambda)\nu_{34},\nu^*_{16}\nu_{34}\}$ 
\cite[Theorem A]{Mi},
$\eta_{16}\bar{\kappa}_{17}\nu_{37}=\bar{\kappa}_{16}\eta_{36}\nu_{37}=0$,
$\sigma_{30}\nu_{37}=0$ and
$(E^3\lambda)\nu_{34}=[\iota_{16},\nu^2_{16}]$ \cite[(7.10)]{Mi}, we obtain
\[\begin{split}
\eta^*_{16}\bar{\nu}_{32}
&\in\eta^*_{16}\circ\{\nu_{32},\eta_{35},\nu_{36}\}
=\{\eta^*_{16},\nu_{32},\eta_{35}\}\circ\nu_{37}\subset\pi^{16}_{37}\circ\nu_{37}\\
&=\{\eta_{16}\bar{\kappa}_{17}\nu_{37},\sigma^3_{16}\nu_{37},
(E^3\lambda)\nu_{34}^2,\nu^*_{16}\nu_{34}^2\}
=\{[\iota_{16},\nu^2_{16}], \nu^*_{16}\nu_{34}^2\}.
\end{split}\]
By \cite[Lemma 12.14]{T} and (\ref{etbn}), we have
$H(\eta^*_{16}\bar{\nu}_{32})=\nu^3_{31}=H(\nu^*_{16}\nu^2_{34})$.
Hence, by the EHP-sequence, we have the first assertion and 
$\eta^*_{19}\bar{\nu}_{35}=\nu^*_{19}\nu^2_{37}$.
By (\ref{n*19x}) and (\ref{n9xi}), we have
$$
[\iota_{19},\nu_{19}]=\nu^*_{19}\nu_{37}+\sigma^3_{19}.
$$
Therefore we obtain 
$[\iota_{19},\nu^2_{19}]=\nu^*_{19}\nu^2_{37}$.
This leads to the second assertion and completes the proof.
\end{proof}

\section{Proof of Theorem \ref{main2}}

We recall the element \cite[p. 187]{Mi}
\begin{equation}\label{sg*}
\sigma^*_{16}\in\{\sigma_{16},2\sigma_{23},\sigma_{30}\}_1. 
\end{equation}

We also recall \cite[II-Proposition 2.1(4)]{Od}:
\begin{equation}\label{sg14ro}
4\sigma_{14}\rho_{21}=2\sigma_{15}\rho_{22}=\sigma_{17}\rho_{24}=0.
\end{equation}

We recall from \cite[I-Proposition 3.4(6)]{Od} the relation
$$
\{\eta_{15},2\sigma_{16},\sigma_{23}\}\ni\omega_{15}+{\eta^*}'
\ \bmod \ \eta_{15}\circ\pi^{16}_{31}+\pi^{15}_{24}\circ\sigma_{24}.
$$
Since $\pi^{16}_{31}=\{\rho_{16},\bar{\varepsilon}_{16},[\iota_{16},\iota_{16}]\}$
\cite[Theorem 10.10]{T},
$\eta_{15}\rho_{16}=\mu_{15}\sigma_{24}=\sigma_{15}\mu_{22}$
\cite[Proposition 12.20.i)]{T}, 
$\eta_{15}\bar{\varepsilon}_{16}=\nu_{15}\sigma_{18}\nu^2_{25}=0$
(\cite[Lemma 12.10]{T}, (\ref{nu10sg})) and 
$\eta_{15}\circ[\iota_{16},\iota_{16}]=[\iota_{15},\eta^2_{15}]=0$ \cite[Theorem]{H},
we have $\eta_{15}\circ\pi^{16}_{31}=\{\sigma_{15}\mu_{22}\}$.
Since $\pi^{15}_{29}=\{\nu_{15}^3, \mu_{15}, \eta_{15}\varepsilon_{16}\}$
\cite[Theorem 7.2]{T}, (\ref{nu10sg}) and (\ref{bnu6sg}),
we have $\pi^{15}_{24}\circ\sigma_{24}=\{\sigma_{15}\mu_{22}\}$. 
Hence, we obtain
\begin{equation}\label{a0}
\{\eta_{15},2\sigma_{16},\sigma_{23}\}\ni\omega_{15}+{\eta^*}'\ \bmod\ \sigma_{15}\mu_{22}.
\end{equation}

We show 
\begin{lem}\label{a1}
$\bar{\nu}_7{\eta^*}'=0$. 
\end{lem}
\begin{proof}
Since $E:\pi_{31}^{7}\to\pi_{32}^{8}$ is a monomorphism \cite[Theorem 1.1(a)]{MMO},
it suffices to show $\bar{\nu}_8(E{\eta^*}')=0$. By (\ref{et*'}) and (\ref{bnu6sg}), we have 
\[
\bar{\nu}_8(E{\eta^*}')\in\bar{\nu}_8\circ-\{\sigma_{16},4\sigma_{23},\eta_{30}\}=\{\bar{\nu}_8,\sigma_{16},4\sigma_{23}\}\circ\eta_{31}.
\]
We observe
$$
\{\bar{\nu}_8,\sigma_{16},4\sigma_{23}\}\subset\{\bar{\nu}_8,2\sigma_{16},2\sigma_{23}\}\supset\{0,\sigma_{16},2\sigma_{23}\}\ni 0\ \bmod\ 
\bar{\nu}_8\circ\pi^{16}_{31}+\pi^8_{24}\circ 2\sigma_{24}.
$$
By relations $\bar{\nu}_8\rho_{16}=0$ \cite[I-Proposition 3.1(4)]{Od},
$\bar{\varepsilon}_{16}=\eta_{16}\kappa_{17}$ (\ref{bep}),
$\bar{\nu}_8\eta_{16}=\nu^3_8$ \eqref{etbn},
$\nu^2_{11}\kappa_{17}=4\bar{\kappa}_{11}$ \cite[Theorem 15.4]{MT} and
$\bar{\nu}_{8}\circ[\iota_{16},\iota_{16}]
 =[\iota_{8},\iota_{8}]\circ \bar{\nu}_{15}^2=0$ \cite[Proposition 2.8(2)]{Og}, we have
$\bar{\nu}_8\circ\pi^{16}_{31}=\{4\nu_8\bar{\kappa}_{11}\}$. 
Since $2\pi^8_{24}=0$, we have $\pi^8_{24}\circ 2\sigma_{24}=0$. 
Hence, we obtain
$\bar{\nu}_8(E{\eta^*}')\subset
 \{4\nu_8\bar{\kappa}_{11}\}\circ\eta_{31}=0$.
This completes the proof. 
\end{proof}

Next, we show
\begin{lem}\label{bnuom7}
$\bar{\nu}_6\omega_{14}\equiv P(\lambda+\xi_{13})\circ\eta_{29}\ \bmod\ \nu_6\sigma_9\kappa_{16},4\bar{\zeta}'_6$. 
\end{lem}
\begin{proof}
Apply \cite[Corollary 5.10]{Mi} to the case $\alpha=\bar{\nu}_6,\ \beta=\omega_{14},\ n=5,\ p=13,\ i=29$. Then, we have
$$
H(\bar{\nu}_6\omega_{14})\equiv\nu_{11}\omega_{14}+\bar{\nu}^2_{11}\nu_{27}\ \bmod\ G,
$$
where 
$$
G=\sum^6_{k=3}{f_k}_*\pi^{5k+1}_{30}+\Ker\{ E: \pi^{10}_{29}\to\pi^{11}_{30}\}.
$$
We know $\bar{\nu}_{19}\nu_{27}=0$ (\ref{bnun}) and $\Ker\{ E: \pi^{10}_{29}\to\pi^{11}_{30}\}=0$. 
We also have $f_3\in\pi^{11}_{16}=0$ and $\pi^{5k+1}_{30}=0$ for $k=5,6$. 
Hence, we obtain
$$
H(\bar{\nu}_6\omega_{14})\equiv\nu_{11}\omega_{14}\ \bmod\ {f_4}_*\pi^{21}_{30}.
$$
Since $\pi^{11}_{21}\circ\pi^{21}_{30}=\{4\bar{\zeta}_{11}\}$, we get that 
$$
H(\bar{\nu}_6\omega_{14})\equiv\nu_{11}\omega_{14}\ \bmod\ 4\bar{\zeta}_{11}.
$$
By relations (\ref{n11om}), Lemma \ref{HPx13} and 
$H(\bar{\zeta}'_6)\equiv \bar{\zeta}_{11}\ \bmod\ 2\bar{\zeta}_{11}$ \cite[(3.8)]{MMO},
we have
\[\begin{split}
H(\bar{\nu}_6\omega_{14})&\equiv\nu_{11}\omega_{14}=\lambda'\eta_{29}+\xi'\eta_{29}\\
&=H(P(\lambda\eta_{31}))+H(P(\xi_{13}\eta_{31})) \ \bmod H(4\bar{\zeta}'_6).
\end{split}\]
 Therefore, we see that 
$$
\bar{\nu}_6\omega_{14}\equiv P(\lambda\eta_{31})+P(\xi_{13}\eta_{31})+4a\bar{\zeta}'_6 \ \bmod\ E\pi^5_{29}
$$
for $a=$ or $a=1$.
By (\ref{sg'nu}) and $2\kappa_{19}=0$, we have
 $\nu_9\sigma_{12}\kappa_{19}=(E^2\sigma')\nu_{16}\kappa_{19}=
2\sigma_9\nu_{16}\kappa_{19}=0$. We know $E\pi^5_{29}=\{\delta_6,\bar{\mu}_6\sigma_{23},\nu_6\sigma_9\kappa_{16}\}\cong(\Z_2)^3,\ E^6\pi^5_{29}=\{\delta_{11},\bar{\mu}_{11}\sigma_{28}\}\cong(\Z_2)^2,\ \bar{\nu}_{11}\omega_{19}=0$ (\ref{bnu9omg}). 
Hence, we conclude that 
\[
\bar{\nu}_6\omega_{14}\equiv P(\lambda\eta_{31})+P(\xi_{13}\eta_{31})
=P(\lambda+\xi_{13})\eta_{29}\ \bmod\ \nu_6\sigma_9\kappa_{16},4\bar{\zeta}'_6.
\]
This completes the proof. 
\end{proof}

By (\ref{etbn}), (\ref{bnu6sg}), (\ref{a0}) and Lemma \ref{a1}, we have
\[\begin{split}
\bar{\nu}_7\omega_{15}
 &\in\bar{\nu}_7\circ\{\eta_{15},2\sigma_{16},\sigma_{23}\}
 \subset\{\nu^3_7,2\sigma_{16},\sigma_{23}\}\\
 &\supset\nu_7\circ\{\nu^2_{10},2\sigma_{16},\sigma_{23}\}
 \ \bmod\ \nu^3_7\circ\pi^{16}_{31}+\pi^7_{24}\circ\sigma_{24}.
\end{split}\]
By relations (\ref{nu11ro}), (\ref{bep}) and $[\iota_{13},\nu_{13}]=0$, we have
$$
\nu_{13}\circ\pi^{16}_{31}
 =\nu_{13}\circ\{\rho_{16},\bar{\varepsilon}_{16},[\iota_{16},\iota_{16}]\}=0
\ \text{ and }\ \nu^3_7\circ\pi^{16}_{31}=0.
$$
We recall $\pi^7_{24}=\{\sigma'\eta_{14}\mu_{15},\nu_7\kappa_{10},\bar{\mu}_7,\eta_7\mu_8\sigma_{17}\}$. By relations (\ref{et9sg}),
$\bar{\nu}_6\mu_{14}=0$ (\ref{mepbn2}) and
$\sigma'\varepsilon_{14}\mu_{22}=\bar{\zeta}'_7$ \cite[(5.10)]{MMO}, we obtain 
$$
\sigma'\eta_{14}\mu_{15}\sigma_{24}=\sigma'\eta_{14}\sigma_{15}\mu_{22}=\sigma'(\bar{\nu}_{14}+\varepsilon_{14})\mu_{22}=\bar{\zeta}'_7. 
$$
Hence, by relations $\nu_7\kappa_{10}\sigma_{24}=0$ (\ref{k7s})
and $\eta_7\mu_8\sigma^2_{17}=0$ (\ref{mu3sg2}), we have
$\pi^7_{24}\circ\sigma_{24}=\{\bar{\zeta}'_7,\bar{\mu}_7\sigma_{24}\}$ and 
 %\begin{equation}\label{b0}
$$
\bar{\nu}_7\omega_{15}\in\nu_7\circ\{\nu^2_{10},2\sigma_{16},\sigma_{23}\}\ \bmod\ \bar{\zeta}'_7,\bar{\mu}_7\sigma_{24}.
$$
By relations (\ref{tn13sg}), (\ref{nulm}), (\ref{n9xi}), 
$2\kappa_{17}=0$, $2\sigma_{17}^2=0$ and (\ref{k7s}), 
for some odd integer $x$, we obtain 
\[\begin{split}
\{\nu^2_{10},2\sigma_{16},\sigma_{23}\}
&\supset\nu_{10}\circ\{\nu_{13},2\sigma_{16},\sigma_{23}\}=
\nu_{10}\xi_{13}+x\nu_{10}\lambda+2x\nu_{10}\xi_{13}\\
&=\sigma^3_{10}+\sigma_{10}\kappa_{17} 
 \bmod\ \nu^2_{10}\circ\pi_{31}^{16}+\pi_{24}^{10}\circ\sigma_{24}=\{\sigma^3_{10}\}.
\end{split}\]
By the relations (\ref{sg'nu}), $2\sigma_{17}^2=0$ and (\ref{nu10sg}), we have
$\nu_{7}\sigma_{10}^3=\sigma'\nu_{14}\sigma_{17}^2=0$.
Hence, we obtain
$$
\bar{\nu}_7\omega_{15}\equiv\nu_7\sigma_{10}\kappa_{17}\ \bmod\ \bar{\zeta}'_7,\bar{\mu}_7\sigma_{24}.
$$
The first $3$ elements become trivial and the last survives in the stable range. This induces the relation
$$
\bar{\nu}_7\omega_{15}\equiv\nu_7\sigma_{10}\kappa_{17}\ \bmod\ \bar{\zeta}'_7.
$$
On the other hand, by Lemma \ref{bnuom7}, we have
\[
\bar{\nu}_7\omega_{15}\equiv 0 \ \bmod\ \nu_7\sigma_{10}\kappa_{17}.
\]
Hence we obtain the equation $\bar{\nu}_7\omega_{15}=\nu_7\sigma_{10}\kappa_{17}$.
This completes the proof of Theorem \ref{main2}.
%By Lemmas \ref{bnuom7}, \ref{b1}, we get Theorem \ref{main2}(2). 

\begin{center}
\textsc{Toshiyuki\ Miyauchi}

\textsc{Department of Applied Mathematics}

\textsc{Faculty of Science}

\textsc{Fukuoka University}

\textsc{Fukuoka, 814-0180, Japan}

\it{E-mail : miyauchi@math.sci.fukuoka-u.ac.jp}
\end{center}

\begin{center}
\textsc{Juno \ Mukai}

\textsc{Developmental Education Center}

\textsc{Matsumoto University}

\textsc{Matsumoto, Nagano Pref., 390-1295, Japan}

\it{E-mail: juno.mukai@matsu.ac.jp}
\end{center}


\begin{thebibliography}{99}

\bibitem{Ba} \textsc{M. G. Barratt}: Note on a formula due to Toda, J. London Math. Soc. \textbf{36}(1961), 95-96.
\bibitem{Bo} \textsc{R.\ Bott}: The stable homotopy of the classical groups,
Ann.\ of Math.\ {\bf 70} (1959), 313--337.
\bibitem{H}
\textsc{P.J. Hilton},
A note on the $P$-homomorphism in homotopy groups of spheres,
Proc.\ Camb.\ Phil.\ Soc.\ {\bf 59} (1955) 230--233.
\bibitem{HKM}
\textsc{Y. Hirato}, \textsc{H. Kachi} and \textsc{J. Mukai}: $21$-st and $22$-nd  homotopy groups of the $n$-th rotation group, \ J. Fac. Sci. Shinshu Univ.\ \textbf{41} (2006), 1--28.
\bibitem{GM}
\textsc{M. Golasi\'nski} and \textsc{J. Mukai}: Gottlieb groups of spheres, 
Topology {\bf 47} (2008), 399--430.
%\bibitem{HM}
%\textsc{Y. Hirato} and \textsc{J. Mukai}: Some Toda bracket in $\pi^S_{26}(S^0)$, \ Math. J. Okayama Univ.\ \textbf{42} (2000), 83--88.
\bibitem{IM}
\textsc{T. Inoue} and \textsc{J. Mukai}: A note on the Hopf homomorphism of a Toda bracket and its application. Hiroshima Math. J. {\bf 33} (2003), 379--389.
%\bibitem{Ka} \textsc{H. Kachi}: On the homotopy groups of rotation groups $R_{n}$, J. Fac. Sci., Shinshu Univ. {\bf 3} (1968), 13--33.
%\bibitem{KM}
%\textsc{H. Kachi} and \textsc{J. Mukai}: Some homotopy groups of the rotation group, \ J. Math. Hiroshima Univ.\ \textbf{41} (1999), 327--345.
\bibitem{Ke} \textsc{M. A. Kervaire}: Some nonstable homotopy groups of Lie groups, \ Illinois J. Math.\ \textbf{4} (1960), 161--169.
\bibitem{Mi} \textsc{M. Mimura}: On the generalized Hopf homomorphism and the higher composition. Part I; II. $\pi_{n+i}(S^n)$ for $i = 21$ and $22$, J. Math. Kyoto Univ.{\bf 4}(1964-5), 171-190; 301--326.
%\bibitem{Mi1} \textsc{M. Mimura}: The homotopy groups of Lie groups of low rank, J.\ Math.\ Kyoto Univ.\ {\bf 6}-2 (1967), 131--176.
\bibitem{MMO}
\textsc{M. Mimura}, \textsc{M. Mori} and \textsc{N. Oda}: Determinations of
$2$-components of the $23$- and $24$-stems in homotopy groups of spheres,
Mem. Fac. Sci. Kyushu Univ. {\bf 29} (1975), 1--42.
\bibitem{MT} \textsc{M. Mimura} and \textsc{H. Toda}: The $(n+20)$-th homotopy groups of $n$-spheres, J. Math. Kyoto Univ.{\bf 3}-1 (1963), 37--58.
\bibitem{MM} \textsc{K. Morisugi} and \textsc{J. Mukai}: Whitehead square of a lift of the Hopf map to a mod $2$ Moore space, J. Math. Kyoto Univ. {\bf 42}-2 (2002), 331--336.
%\bibitem{Mu1} \textsc{J. Mukai}: On the stable homotopy of a $Z_2$-Moore space, Osaka J. Math.
%\textbf{6}(1969), 63--91. 
\bibitem{Mu2}
\textsc{J. Mukai}: Determination of the $P$-image by Toda brackets, Geometry and Topology Monographs 
\textbf{13}(2008), 355--383. 
\bibitem{Od} \textsc{N. Oda}: Unstable homotopy groups of spheres, Bull. Inst. Adv. Res. Fukuoka Univ. {\bf 44}(1979), 49--152.
\bibitem{Od1} \textsc{N. Oda}: On the orders of the generators in the $18$-stem of the homotopy groups of spheres, Adv. Stud. Pure Math. {\bf 9}(1986), 231--236.
\bibitem{Og} \textsc{K.\ {\^O}guchi}: Generators of $2$-primary components of homotopy groups and symplectic groups, J. Fac. Sci. Univ. Tokyo  \textbf{11} (1964), 65--111. 
\bibitem{Th} \textsc{S. Thomeier}: Whitehead products and homotopy groups of spheres, Proc. 13-th Biennial Seminar of Can. Math. Congress, {\bf 2} (1972), 144--155. 
\bibitem{T}
\textsc{H. Toda}: Composition methods in homotopy groups of spheres, Ann. of Math. Studies, {\bf 49}, Princeton, 1962.
\bibitem{T1} \textsc{H. Toda}: A survey of homotopy theory. Advances in Math. {\bf 10} (1973), 417--455. 
\bibitem{WG} \textsc{G.W. Whitehead}: A generalization of the Hopf invariant, Ann. of Math. {\bf 51} (1950), 192--237.
\end{thebibliography}
\end{document}